\documentclass{article}
\usepackage{etoolbox}
\usepackage{natbib}
\makeatletter
\providecommand{\institute}[1]{
	\apptocmd{\@author}{\end{tabular}
	\par \small
	\begin{tabular}[t]{c}
		#1}{}{}
}

\makeatother
\usepackage{amsthm}
\newcommand{\smartqed}{}
\newcommand{\journalname}[1]{}
\newcommand{\at}{, }
\newcommand{\email}[1]{\href{mailto:#1}{#1} }
\newcommand{\qedperso}{}
\newtheorem{definition}{Definition}
\newtheorem{remark}{Remark}
\newtheorem{proposition}{Proposition}
\newtheorem{lemma}{Lemma}
\usepackage{bbding}
\newcommand{\Letter}{\Envelope}

\usepackage[utf8]{inputenc}
\usepackage[T1]{fontenc}
\usepackage{amsmath}
\usepackage{amsfonts}
\usepackage{amssymb}
\usepackage{lmodern}
\usepackage{upgreek}
\usepackage[left=2cm,right=2cm,top=2cm,bottom=2cm]{geometry}
\usepackage{dsfont}
\usepackage{color}
\usepackage{tikz}
\usepackage{datetime}
\usepackage[caption=false]{subfig} 
\usepackage{sublabel}
\usepackage[normalem]{ulem} 
\usepackage{placeins} 
\usepackage{enumitem} 
\setlist[enumerate]{leftmargin=10mm}
\usepackage{natbib} 
\usepackage{comment} 
\usepackage[hidelinks]{hyperref}

\hypersetup{hidelinks}

\smartqed

\journalname{Mathematical Geoscience}

\bibpunct{(}{)}{,}{a}{}{;}

\newcommand{\Esp}[1]{\mathrm{E}\left[#1\right]}
\newcommand{\Var}[1]{\mathrm{V}\left[#1\right]}
\newcommand{\Cov}[1]{\mathrm{Cov}\left[#1\right]}
\newcommand{\norm}[1]{{\Vert{}#1\Vert}}
\newcommand{\abs}[1]{{\vert{}#1\vert}}
\newcommand{\transpose}[1]{{#1}^t}

\newcommand{\sachant}{|}

\newcommand{\R}{\mathds{R}}

\newcommand{\set}[1]{{\lbrace{#1}\rbrace}}

\newcommand{\accolade}[1]{\left\{{\renewcommand{\arraystretch}{1.15}\begin{array}{lcl}#1\end{array}}\right.}

\newcommand{\accoladesplit}[1]{\left\{\begin{aligned}#1\end{aligned}\right.}

\newcommand{\virguleacc}{\rlap{\,,}}
\newcommand{\pointacc}{\rlap{\,.}}
\newcommand{\ie}{i.e.\@ }

\newcommand{\tcpx}[1]{\;}

\newcommand{\M}{M}
\newcommand{\Y}{Y}
\newcommand{\agg}{{\mathcal{A}}}
\newcommand{\full}{{full}}
\newcommand{\vv}{v}

\makeatletter
\newcommand{\shorteq}{%
  \settowidth{\@tempdima}{a}
  \, \resizebox{\@tempdima}{\height}{=} \,%
}
\makeatother

\newcommand{\did}[1]{{\color{teal}#1}} 
\newcommand\nicolas[1]{{\color{teal}#1}}
\newcommand\fb[1]{{\color{teal}#1}} 

\renewcommand{\did}[1]{#1}
\renewcommand\nicolas[1]{#1}
\renewcommand\fb[1]{#1}

\newcommand{\NK}[1]{nested Kriging#1}
\newcommand{\MAJNK}[1]{Nested Kriging#1}

\newtheorem{assumption}{Assumption}

\title{Properties and comparison of some Kriging sub-model aggregation methods\thanks{Part of this research was conducted within the frame of the Chair in Applied Mathematics OQUAIDO, which gathers partners in technological research (BRGM, CEA, IFPEN, IRSN, Safran, Storengy) and academia (Ecole Centrale de Lyon, Mines Saint-\'Etienne, University of Nice, University of Toulouse and CNRS) around advanced methods for Computer Experiments. \did{The authors F.Bachoc and D.Rulli\`{e}re acknowledge support from the regional MATH-AmSud program, grant number 20-MATH-03. The authors are grateful to the Editor-in-Chief, an Associate Editor and a referee for their constructive suggestions that lead to an improvement of the manuscript.}}
}

\author{Fran\c{c}ois Bachoc         \and
        Nicolas Durrande  \and
        Didier Rulli\`ere \and
        Cl\'ement Chevalier
}

\institute{Fran\c{c}ois Bachoc (\Letter)\at
              Institut de Math\'ematiques de Toulouse, Universit\'e Paul Sabatier, Toulouse, France \\
              \email{francois.bachoc@math.univ-toulouse.fr} 
           \and
           Nicolas Durrande\at
              Secondmind, Cambridge, UK \\
              \email{nicolas@secondmind.ai} %
           \and
            Didier Rulli\`ere\at
            Mines Saint-Etienne, UCA, CNRS, UMR 6158 LIMOS, Institut Henri Fayol, Saint-Etienne, France\\
            \email{didier.rulliere@emse.fr} %
                         \and
           Cl\'ement Chevalier\at
              Institute of Statistics, University of Neuch\^atel, Neuch\^atel, Switzerland. \\
              \email{clement.chevalier@unine.ch} %
}

\date{}

\begin{document}

\maketitle

\begin{abstract}
Kriging is a widely employed technique, in particular for computer experiments, in machine learning or in geostatistics. An important challenge for Kriging is the computational burden when the data set is large. \nicolas{This article} focuses on a class of methods aiming at decreasing this computational cost, consisting in aggregating Kriging predictors based on smaller data subsets. It proves that aggregation methods \nicolas{that ignore the covariance between sub-models} can yield an inconsistent final Kriging prediction. In contrast, \nicolas{a theoretical study} of the \NK{} method shows additional attractive properties for it: First, this predictor is consistent, second it can be interpreted as an exact conditional distribution for a modified process and third, the conditional covariances given the observations \nicolas{can be computed efficiently}. \nicolas{This article also includes a theoretical and numerical analysis of} how the assignment of the observation points to the sub-models can affect the prediction ability of the aggregated model. \fb{Finally, the \NK{} method} \nicolas{is extended} \fb{to measurement errors and to universal Kriging.}
\end{abstract}

\section{Introduction} \label{section:intro}

Kriging \fb{(\cite{krige1951statistical, matheron70theorie}, see also \citep{Cre1993,stein2012interpolation,santner2013design})} consists in inferring the values of a Gaussian random field given observations at a finite set of observation points.
It has become a popular method for a
large range of applications, such as geostatistics \citep{matheron70theorie}, numerical code approximation \citep{sacks89design,santner2013design,bachoc16improvement}, global optimization \citep{jones98efficient} or machine learning \citep{Rasmussen2006}.

\nicolas{Let $Y$ be a centered Gaussian process on $D \subset \R^d$, with covariance function $k: D \times D \to \R$ (ie $k(x,x')=\Cov{\Y(x),\Y(x')}$), and let $x_1,...,x_n \in D$ be $n$ points in the input space where $Y$ is observed (exactly). The assumptions that $Y$ is zero-mean (simple Kriging, corresponding to a known mean) and that it is observed without observation noise are common in the literature and they will be used in throughout the paper for conciseness. These two assumptions will however be relaxed in Sects. \did{\ref{subsection:measurement:errors} and \ref{subsection:universal:kriging}} to ensure the method under study can be applied in more practical case studies.}

\nicolas{Let} $X$ be the $n \times d$ matrix with row $i$ equal to $\transpose{x_i} $.
For any functions $f: D \rightarrow \R$, $g: D \times D \rightarrow \R$ and for any matrices $A=\transpose{(a_1,\ldots, a_n)}$ and $B=\transpose{(b_1, \ldots, b_m)}$, with $a_i \in D$ for $i=1,...,n$ and $b_i \in D$ for $i=1,...,m$, $f(A)$ \nicolas{denotes} the $n \times 1$ real valued vector with components $f(a_i)$ and \did{ $g(A,B)$ denotes} the $n \times m$ real valued matrix with components $g(a_i, b_j)$, $i=1,\dots, n$, $j=1,\dots, m$. 
With this notation, the conditional distribution of $\Y$ given the $n \times 1$ vector of observations $\Y(X)$ is Gaussian with mean, covariance and variance:
\begin{equation}
\accoladesplit{
	\M_\full(x) &= \Esp{\Y(x)|\Y(X)} = k(x,X) k(X,X)^{-1} \Y(X) \virguleacc \\
	c_\full(x,x') &= \Cov{\Y(x),\Y(x')|\Y(X)} = k(x,x') - k(x,X) k(X,X)^{-1} k(X,x') \virguleacc\\
v_\full(x) &= c_\full(x,x) \pointacc
}
\label{eq:fullmodel}
\end{equation}


\nicolas{Computing the terms on the right hand side of~\eqref{eq:fullmodel} requires} to invert the $n \times n$ covariance matrix $k(X,X)$, which leads to a $O(n^2)$ \did{storage requirement} and $O(n^3)$ \did{computational complexity}. In practice, this \nicolas{posterior} distribution is hence difficult to compute when the number of observation points exceeds a few thousands. \fb{The challenge of a large number of observation points for Kriging is for instance acknowledged in Section 5 of \cite{davis2019development}.}

Many methods have been proposed in the literature to approximate the conditional distribution~\eqref{eq:fullmodel}, \nicolas{without incurring a large} computational cost. These methods include low rank approximation\nicolas{s} \citep[see][and the references therein for a review]{stein14limitations}, sparse \nicolas{methods} \citep{hensman2013}, covariance tapering \citep{furrer2006covariance,kaufman08covariance}, Gaussian Markov Random Fields \citep{rue05gaussian, datta2016hierarchical}, and aggregation-based approximations \citep{deisenroth2015}. \nicolas{This paper focuses on the later approach, which consists of building sub-models based on subsets of the data before aggregating their predictions. More precisely,} these methods first construct $p$ sub-models $M_1,...,M_p: D \to \R$, where $M_i(x)$ is a predictor of $Y(x)$ built from a subset $X_i$ of size $n_i \times d$ of the observation points in $X$. The rationale is that when $n_i$ is small compared to~$n$, $M_i(x)$ can be obtained \did{efficiently} with a small computational cost. The sub-models $M_1,...,M_p$ are then combined to obtain the aggregated predictor $M_{\agg}: D \to \R$. Examples of aggregation techniques for Gaussian processes are (generalized) products of experts and (robust) Bayesian committee machines \citep{hinton2002training,trespBCM,caoGPOE,deisenroth2015,van2015optimally},
as well as the \NK{} predictor (or Nested Pointwise Aggregation of Experts (NPAE)) \citep{rulliere2018nested}.  It must be noted that \NK{} relies on a particular aggregation of several predictors; \did{a review of probability aggregation methods in Geoscience can be found in \citet{allard2012probability}}. 

\nicolas{Aggregation methods can be of particular interest in a geoscience context.} \did{As it is well known, the origins of Kriging are directly linked to mining and geostatistics \citep{cressie1990origins, chiles2018fifty}, and \nicolas{it is common to encounter large datasets in such context}. As an example, imagine a measurement (say, radioactivity) at several locations on the ground. The measures can be done with simple movable devices at many locations, eventually repeated at several times, so that the number of measurements can be important, and each measure may come with a measurement error. There is a necessity to handle a large amount of potentially noisy measures, which is not possible with classical Kriging techniques, but becomes possible with \nicolas{aggregation techniques} such as \NK{}. \nicolas{Despite its novelty,} the \NK{} predictor has already been used in several application fields, including earth and geostatistical sciences: see \citet{sun2019} for the study of air pollution,  \citet{bacchi2020} for tsunami analysis and \citet{krity2020} for contaminant source localization in the ground. This also emphasizes the importance of aggregation methods in a geostatistical context.
}

\did{Benchmarks of different spatial interpolation methods are also of importance, in particular when dealing with big data. Among recent ones, a general benchmark on some other methods applicable with big data can be found in \cite{heaton2019case}.
\nicolas{It is worth noting that \NK{} often appears among the two or three best competitors (typically among around 12 methods) in the numerical studies that include it into the benchmark} \citep{rulliere2018nested, liu2018generalized, he2019query, liu2020, van2020cluster}. \nicolas{These good empirical performances are supported by the theoretical properties of \NK{} which guaranty some} optimal performances under a correct estimation of the underlying hyperparameters and under stationarity.
}

\nicolas{This paper provides additional theoretical insights into aggregation methods, with an emphasis on \NK{} \citep{rulliere2018nested}. A distinction is introduced between aggregation techniques that only rely on the conditional variances $v_1(x),...,v_p(x)$ of the sub-models (such as products of expert and Bayesian committee machines), and the ones, like \NK{}, where the aggregation accounts for the covariance between the sub-models. As shown in Proposition~\ref{prop:non:consistency}, techniques based only on the sub-model variances can lead to inconsistent estimators of $Y(x)$} \fb{in the infill (fixed-domain) asymptotic setting \citep{Cre1993,stein2012interpolation}}.
\nicolas{On the other hand, Proposition~\ref{prop:consistency} guaranties the consistency, again in the infill (fixed-domain) asymptotic setting,} of the \NK{} predictor. In addition, the \NK{} predictor can be interpreted as an exact conditional expectation, for a slightly different Gaussian process prior. 

\nicolas{Furthermore, the paper introduces two extensions of the \NK{} methodology which broaden the use cases where the approach can be applied. The first one is to make \NK{} amenable to observation noise corrupting the measurements (the initial exposition in \citet{rulliere2018nested} focused on the noiseless setting). The second is to generalise the method to universal Kriging, where the Gaussian process prior includes an unknown mean function that must be estimated. Note that both generalisations result in similar storage or computational requirements as the original approach.}

The structure of the article is as follows. Section~\ref{section:non:consistency} introduces covariance-free aggregation techniques and present the non-consistency result. Section~\ref{section:rulliere} summarizes the aggregation method of \citet{rulliere2018nested}, gives its consistency property, shows how it can be interpreted as an exact conditional expectation and provides some error bounds for the \NK{} approximation. It also provides a numerical illustration of the consistency and inconsistency properties shown in this paper. Section~\ref{section:impact:clustering} studies the impact of the assignment of the observation points to the sub-models. \nicolas{Finally, Sect.~\ref{section:extensions} provides} \fb{the extensions to measurement errors and universal Kriging} \nicolas{and concluding remarks are given in Sect.~\ref{section:conclusion}}. \nicolas{For the sake of the clarity of the exposition,} most of the proofs are postponed to the appendix.

\section{Covariance-free aggregation techniques} \label{section:non:consistency}

For $i=1,...,p$, let $X_i$ be a $n_i \times d$ matrix composed of a subset of the lines of $X$\nicolas{, such that} $n_1+ \cdots +n_p = n$ and $X_1,...,X_p$ constitute a partition of $X$. \nicolas{For $i=1,...,p$, let $M_i(x) = k(x,X_i) k(X_i,X_i)^{-1} Y(X_i) $ and $v_i(x) = k(x,x) - k(x,X_i) k(X_i,X_i)^{-1} k(X_i,x)$ be the conditional mean and variance of $Y(x)$ given $Y(X_i)$. This section focuses on aggregated predictors that only depend on (predicted) variances}
\begin{equation} \label{eq:form:aggregation:variance}
M_{\agg}(x) = \sum_{k=1}^{p} \alpha_{k}(\vv_1(x),...,\vv_{p}(x),\vv_{prior}(x)) M_k(x),
\end{equation}
where $\vv_{prior}(x) = k(x,x)$ and with $\alpha_k: [0,\infty)^{p+1}\to \R$. \nicolas{Several aggregation techniques, such as} product of expert (POE), generalized product of expert (GPOE), Bayesian committee machines (BCM) and robust Bayesian committee machines (RBCM), \nicolas{can be written under  the form of~\eqref{eq:form:aggregation:variance}}. For POE \citep{hinton2002training,deisenroth2015} and GPOE \citep{caoGPOE} \nicolas{the weights associated to each sub-model are}
\[
\alpha_k(v_1,...,v_p,v_{prior}) = 
\frac{\beta_k(x) \frac{1}{v_k}
}
{
\sum_{i=1}^{p} \beta_i(x) \frac{1}{v_i} 
}
\]
with $\beta_i(x) = 1$ for POE and $\beta_i(x) = (1/2) [ \log(\vv_{prior}(x)) - \log(\vv_i(x)) ]$ for GPOE. For BCM \citep{trespBCM} and RBCM \citep{deisenroth2015} \nicolas{they are}
\[
\alpha_k(v_1,...,v_p,v_{prior}) = 
\frac{\beta_k(x) \frac{1}{v_k}
}
{
\sum_{i=1}^{p} \beta_i(x) \frac{1}{v_i} 
+\left(1 - \sum_{i=1}^{p} \beta_i(x) \right) \frac{1}{v_{prior}} 
}
\]
with $\beta_i(x) = 1$ for BCM and $\beta_i(x) = (1/2) [ \log(\vv_{prior}(x)) - \log(\vv_i(x)) ]$ for RBCM.

\nicolas{The next proposition shows} that aggregations given by~\eqref{eq:form:aggregation:variance} can lead to mean square prediction errors that do not go to zero as $n \to \infty$, when considering triangular arrays of observation points that become dense in a compact set $D$ \fb{(which is the infill asymptotic setting, \cite{Cre1993,stein2012interpolation})}. This proposition thus provides a counter-example, but does not prove that aggregation procedures given by \eqref{eq:form:aggregation:variance} are inconsistent in general. This inconsistency will however be confirmed in some further simple numerical experiments. \nicolas{The property relies on} Gaussian processes satisfying the no-empty ball (NEB) property, which has been introduced in \citet{vazquez2010convergence}.
\begin{definition} \label{def:NEB}
A Gaussian process $Y$ on $D$ has the NEB property if for any  $x_0 \in D$ and for any sequence $(x_i)_{i \geq 1}$ of points in $D$, the following two assertions are equivalent.
\begin{enumerate}
\item $\Var{ Y(x_0) | Y(x_1),...,Y(x_n) }$ goes to $0$ as $n \to \infty$,
\item $x_0$ is an adherent point of the sequence $(x_i)_{i \geq 1}$. 
\end{enumerate}
\end{definition}

\begin{proposition}[Non-consistency of some covariance-free aggregations] \label{prop:non:consistency}
Let $D$ be a compact subset of $\mathds{R}^d$ with non-empty interior. Let $Y$ be a Gaussian process on $D$ with mean zero and covariance function $k$. Assume that $k$ is defined on $\mathds{R}^d$, continuous and satisfies $k(x,y)>0$ for two distinct points $x,y$ in the interior of $D$. Assume also that $Y$ has the NEB property.
For $n \in \mathbb{N}$ and for any triangular array of observation points $(x_{ni})_{1 \leq i \leq n; n \in \mathbb{N}}$, let $p_n$ be a number of Kriging predictors, $X$ be the $n \times d$ matrix with row $i$ equal to $\transpose{x}_{ni}$, and $X_{1},...,X_{p_n}$ be a partition of $X$. For $n \in \mathbb{N}$ let $M_{\agg,n}$ be defined as in~\eqref{eq:form:aggregation:variance} with $p$ replaced by $p_n$. Finally, assume that
\begin{eqnarray} \label{eq:assumption:aggregation:for:inconsistency}
\alpha_{k}(\vv_1(x),...,\vv_{p_n}(x),\vv_{prior}(x)) \leq \frac{
a( \vv_k(x) , \vv_{prior}(x) )
}{
\sum_{l=1}^{p_n} b( \vv_l(x) , \vv_{prior}(x) )
},
\end{eqnarray}
where $a$ and $b$ are given deterministic continuous functions from $\Delta = \{ (x,y) \in (0,\infty)^2; x \leq y \}$ to $[0 , \infty)$, with $a$ and $b$ positive on $\mathring{\Delta} = \{ (x,y) \in (0,\infty)^2; x < y \}$.

Then, there exists a triangular array of observation points $(x_{ni})_{1 \leq i \leq n; n \in \mathbb{N}}$ such that \linebreak[4] $\lim_{n \to \infty} \sup_{x \in D} \min_{i=1,...,n} || x_{ni} - x || = 0$, a triangular array of submatrices $X_1,...,X_{p_n}$ forming a partition of $X$, with $p_n \to_{n \to \infty} \infty$ and $p_n / n \to_{n \to \infty} 0$, and such that
\begin{equation} \label{eq:int:pred:large} 
\liminf_{n \to \infty}
\int_{D}
\Esp{\left(Y(x) - M_{\agg,n}(x)\right)^2}
dx
> 0.
\end{equation}
As a consequence, there exists a subset $C$ of $D$ with strictly positive Lebesgue measure so that, for all $x_0 \in C$, 
\begin{equation} \label{eq:pred:large}
 \Esp{\left(Y(x_0) - M_{\agg,n}(x_0)\right)^2}
\not \to_{n \to \infty} 0.
\end{equation}
\label{prop:nonconsistency}
\end{proposition}

It is easy to see that the proposition applies to the POE, GPOE, BCM, RBCM methods introduced above. Hence, Proposition~\ref{prop:nonconsistency} constitutes a significant theoretical drawback for an important class of aggregation techniques in the literature, which are based solely on conditional variances.

The detailed proof is given in \nicolas{Appendix}~\ref{app:nonconsistency}. The intuitive explanation is that the aggregation methods for which the proposition applies ignore the correlations between the different Kriging predictors. Hence, for prediction points around which the density of observation points is smaller than on average, too much weight can be given to Kriging predictors based on distant observation points. 
\nicolas{It is worth noting} that, in the proof of Proposition~\ref{prop:nonconsistency}, the density of observation points in the subset of $D$ where the inconsistency occurs is asymptotically negligible compared to the average density of observation points. Hence, this proof does not apply to triangular arrays of observation points for which the density is uniform (for instance grids of points or uniformly distributed random points). Thus, Proposition~\ref{prop:nonconsistency} does not preclude the consistency of the POE, GPOE, BCM, RBCM methods for uniformly dense observation points.
It should be noted that when doing optimization, or when looking for optimal designs for parameter estimation, see Fig. 4 in~\cite{zhu2006spatial}, one may naturally end up with strongly non-uniform densities of observation points, so that unbalanced designs leading to non-consistency are not purely theoretical.

\begin{remark}
The NEB property holds for many Gaussian processes defined on $D \subset \R^d$ with zero mean function and covariance function $k$. In particular, assume that
$k$ has a positive spectral density (defined by $\hat{k}(\omega) = \int_{\mathds{R}^d} k(x) \exp( - J \transpose{x} \omega ) dx$ with $J^2 = -1$ and for $\omega \in \mathds{R}^d$). Assume that there exist $0 \leq A < \infty$ and $0 \leq T < \infty$ such that $1/\hat{k}(\omega) \leq A (1+||\omega||^t)$, with $||.||$ the Euclidean norm. Then $Y$ has the NEB property \citep{vazquez2010convergence,vazquez10pointwise}.
These assumptions are satisfied by many stationary covariance functions, \nicolas{such as Mat\'ern ones}, but a notable exception is the Gaussian covariance function \citep[Proposition~1 in][]{vazquez10pointwise}.
\end{remark}

\begin{remark}
The partitions $X_1,\ldots,X_{p_n}$ for which the inconsistency occurs in Proposition~\ref{prop:non:consistency} can typically be representative of outputs of clustering algorithms, in the sense that points in the same group $X_i$ would be close to each other. \nicolas{This is further discussed in} Remark~\ref{remark:clustering} \nicolas{in Appendix \ref{app:nonconsistency}}. 
\end{remark}

\begin{remark}
Proposition~\ref{prop:nonconsistency} does not imply that all the aggregation methods based only on the conditional variances are inconsistent. In particular, consider the aggregation consisting in predicting from the subset of observations yielding the smallest conditional variance, defined by $M_{\agg}(x) = M_{i(x)}(x)$ where $i(x) = \mathrm{argmin}_{j=1,...,p} v_j(x)$. Then, the aggregated predictor $M_{\agg}(x)$ can be seen to be consistent from the proof of Proposition~\ref{prop:consistency} below.  
\end{remark}

\section{The \NK{} prediction} \label{section:rulliere}

\nicolas{This section assumes that $M_1(x),...,M_p(x)$ have mean zero and finite variance, but not necessarily that they can be written as} $M_i(x) = k(x,X_i) k(X_i,X_i)^{-1} Y(X_i)$. 
Let $M(x) = \transpose{(M_1(x),...,M_p(x))}$ be the vector of sub-models, $K_M(x)$ be
the $p \times p$ covariance matrix of $(M_1(x),...,M_p(x))$, and $k_M(x)$ be 
the $p \times 1$ vector with component $i$ equal to $\Cov{M_i(x), Y(x)}$. \nicolas{The main assumption that will be required hereafter is:}

\begin{assumption}[Assumptions on sub-models] For all $x\in D$, the random variables $Y(x), M_1(x),...,M_p(x)$ have mean zero and finite variance, and the matrix $K_M(x)=\Cov{M(x), M(x)}$ is invertible. Furthermore, \nicolas{the following assumptions may be considered separately}:
\begin{enumerate}[label=(H\arabic*)]
\item\label{H1} $\M$ is linear in $\Y(X)$: for all $x\in D$, there exists a deterministic $p\times n$ matrix $\Lambda(x)$ such that $\M(x) = \Lambda(x) \Y(X)$, \ie each sub-model is a linear combination of observations $Y(X)$.
\item\label{H2} $\M$ interpolates $\Y$ at $X$: for any component $x_k$ of $X$ there is at least one index $i_k \in \set{1, \ldots, p}$ such that $M_{i_k}(x_k)=Y(x_k)$, \ie any observation is interpolated by at least one sub-model.
\item\label{H3} $(\M,\Y)$ is Gaussian: the joint process $(M_1(x),...,M_p(x), Y(x))_{x \in D}$ is multivariate Gaussian.
\end{enumerate}
\end{assumption}
These assumptions \nicolas{are not particularly restrictive and they are satisfied} in the classical situation where the sub-models are given by interpolating Kriging models $M_i(x) = k(x,X_i) k(X_i,X_i)^{-1} Y(X_i)$, $i \in \set{1, \ldots, p}$. \nicolas{Note that the relaxation of (H2) is takled in Section~\ref{section:extensions} and that several results presented in this section can be extended to the case where (H3) is not satisfied} by using matrix pseudo-inverses.

In \citet{rulliere2018nested}, the aggregated predictor $M_{\agg}(x)$ is defined as the best linear predictor of $Y(x)$ from $M_1(x),...,M_p(x)$, which implies
\begin{equation} \label{eq:aggregation:rulliere}
\accolade{
M_{\agg}(x) &=& \transpose{k_M(x)} K_M(x)^{-1}M(x) \, ,\\
\vv_{\agg}(x) &=& \Esp{ (Y(x) - M_{\agg}(x))^2 } =  k(x,x) -  \transpose{k_M(x)}K_M(x)^{-1}k_M(x) \, .
}
\end{equation}
Under assumptions~\ref{H1},~\ref{H2} and~\ref{H3}, the aggregated predictor $M_{\agg}$ preserves the linearity, the interpolation properties, and the conditional Gaussianity. 
\nicolas{Furthermore, using \ref{H1} one easily gets the expressions of} \did{$k_M(x)$ and $K_M(x)$
\begin{equation}
    \accolade{
    k_M(x) &=& \Lambda(x) k(X,x),\\
    K_M(x) &=& \Lambda(x) k(X,X) \transpose{\Lambda(x)}.
    }
\end{equation}
The aggregated predictor is straightforward to compute in this case, which occurs for example when the submodels $M_1(x), ..., M_p(x)$ are simple Kriging predictors.
}

\citet{rulliere2018nested} show that, for $n$ observation points and $q$ prediction points, the complexity of the aggregation procedure $M_{\agg}$ can reach simultaneously $O(n)$ in \did{storage requirement} and $O(n^2q)$ in \did{computational complexity} when $q=o(n)$.
This \did{computational} complexity is larger than \nicolas{the one of covariance-free aggregation procedures but much smaller than the standard Kriging complexity (see Sect.~\ref{subsection:measurement:errors} for more details)}.
This makes possible the use of this aggregation method with a large number of observations (up to one million points in \citet{rulliere2018nested}). \did{The calculation of the \NK{} predictor can also \nicolas{benefit from parallel computing, both for building the sub-models and for predicting at different prediction points. A} public implementation using parallel computation and allowing measurement errors \fb{(see Sect. \ref{subsection:measurement:errors}) and universal Kriging (see Sect. \ref{subsection:universal:kriging})} is available at \url{https://github.com/drulliere/nestedKriging}.}

\nicolas{Although this article focuses on the case where the covariance function $k$ of the Gaussian process $Y$ is known, 
the parameters of the covariance function often need to be estimated in practice \citep{roustant12dice,abrahamsen97review,stein2012interpolation}}. 
\fb{In a big data context where the aggregated predictor of \eqref{eq:aggregation:rulliere} is relevant, classical parameter estimation methods like maximum likelihood \citep{stein2012interpolation} or cross validation \citep{Bachoc13cross,bachoc2017cross,zhang10kriging} are too computationally prohibitive to be carried out directly. \citet{rulliere2018nested} suggest to apply cross validation to the aggregated predictor in \eqref{eq:aggregation:rulliere} rather than to the full Kriging predictor in \eqref{eq:fullmodel}, and to use stochastic gradient for optimization with respect to the covariance parameters. This results in a procedure that is applicable to a large data set. One could also use a smaller subset of a large data set specifically for covariance parameter estimation by classical maximum likelihood or cross validation. Finally, one could also optimize, with respect to the covariance parameters, the sum of the logarithms of the likelihoods (or of cross validation scores) from each of the subsets $X_1,Y(X_1), \ldots , X_n,Y(X_n)$. This enables to exploit the entire data set for covariance parameter estimation, while keeping a manageable computational complexity.}

\nicolas{The rest of the section focuses on the theoretical properties of this particular aggregation method: it contains the consistency results \fb{under infill asymptotics}, reinterprets the \NK{} approximation as the exact conditional expectation for a modified Gaussian process, and provides bounds on the errors $M_{\agg}(x) - M_{\full}(x)$ and $v_\agg(x)-v_\full(x)$.}


\subsection{Consistency}

\nicolas{The next proposition provides} the consistency result in the case where $\Y$ is a Gaussian process on $D$ with mean zero and $M_i(x) = k(x,X_i)k(X_i,X_i)^{-1}Y(X_i)$, which implies~\ref{H1},~\ref{H2},~\ref{H3}. The proof is given in Appendix~\ref{app:consistency}.

\begin{proposition}[Consistency]
Let $D$ be a compact nonempty subset of $\mathds{R}^d$. Let $\Y$ be a Gaussian process on $D$ with mean zero and continuous covariance function $k$. Let $(x_{ni})_{1 \leq i \leq n, n \in \mathds{N}}$ be a triangular array of observation points so that $x_{ni} \in D$ for all $1 \leq i \leq n, n \in \mathds{N}$ and so that for all $x \in D$, $\lim_{n \to \infty} \min_{i=1,...,n} || x_{ni} - x || = 0$.

For $n \in \mathds{N}$, let $X=(x_{n1},...,x_{nn})^t$, let $\M_{1}(x),...,\M_{p_n}(x)$ be any collection of $p_n$ Kriging predictors based on respective design points $X_1,\ \dots, \ X_{p_n}$, where $X_i$ is a subset of $X$, with $M_i(x) = k(x,X_i)k(X_i,X_i)^{-1}Y(X_i)$ for $i=1,...,p_n$. Assume that each row of $X$ is a row of at least one $X_i$, $1 \leq i \leq p_n$. Then\nicolas{, for $M_{\agg}(x)$ defined as} in~\eqref{eq:aggregation:rulliere}:
\begin{equation}
\sup_{x \in D} \mathbb{E} \left[
\left(
Y(x) - M_{\agg}(x)
\right)^2
\right]
\to_{n \to \infty} 0.
\end{equation}
\label{prop:consistency}
\end{proposition}

Proposition~\ref{prop:consistency} shows that, contrary to several aggregation techniques, taking into account the correlations between the predictors enables the aggregation method of \citet{rulliere2018nested} to have a guaranteed consistency.\\ 

\textbf{Numerical \nicolas{illustration of the consistency results}.} \nicolas{Propositions~\ref{prop:nonconsistency}  and~\ref{prop:consistency} are now illustrated on simple examples where the test functions are given by random samples of a centered Gaussian Process $Y$ with Mat\'ern 3/2 covariance \citep[see][]{Rasmussen2006}}. The observation points $x_1, \ \dots \ , \ x_n \in [0,1]$ are ordered and gathered into groups of $\sqrt{n}$ consecutive points to build $\sqrt{n}$ sub-models.
These sub-models are then aggregated following the various methods presented earlier in order to make predictions $M_{\agg}(x_t)$ at $x_t=0.8$. The criterion used to assess the quality of a prediction is the mean square error: $MSE=\Esp{(Y(x_t)-M_{\agg}(x_t))^2}$. Since the prediction methods that are benchmarked all correspond to linear combinations of the observed values, this expectation can be computed analytically and there is no need to generate actual samples from the test functions. 

\nicolas{Two different settings are considered for the input point distribution and the kernel parameters: (A) a uniform distribution and a lengthscale equal to $0.1$ (Fig.~\ref{fig:consistency}.a) and (B), a beta distribution $\beta (10, 10)$ and a lengthscale of $0.2$ (Fig.~\ref{fig:consistency}.b). In both case the variance of $Y$ is set to one. A small nugget effect ($10^{-9}$ for A and $10^{-10}$ for B) is also included in the sub-models to ensure their computations are numerical stable}. Finally, the experiments are repeated 100 times with different input locations $x_1, \dots, x_n$.
\begin{figure}[htp]
  \centering
  \subfloat[Experiment A settings.]{\label{fig:consistency:unif_settings}%
           \includegraphics[width=8cm]{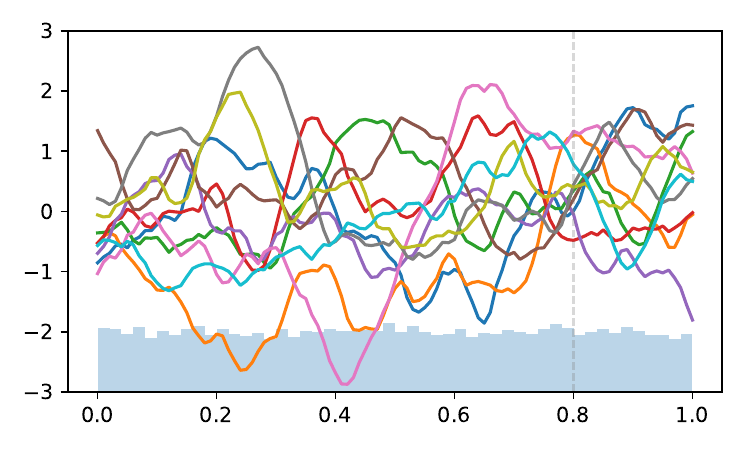}} \qquad
  \subfloat[Experiment B settings.]{\label{fig:consistency:beta_settings}%
           \includegraphics[width=8cm]{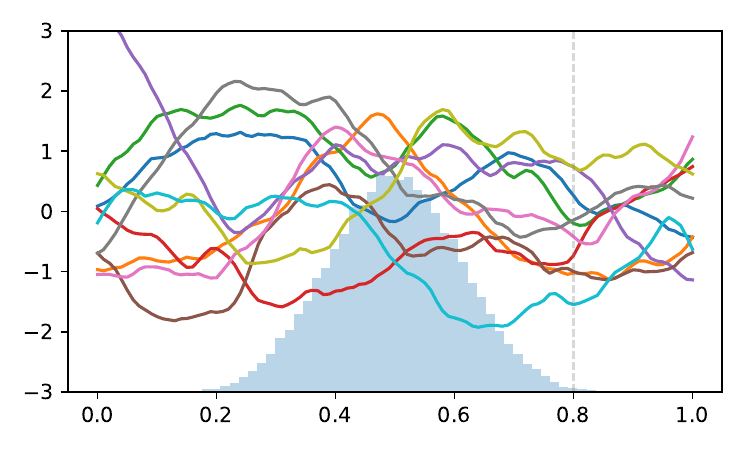}} \\
  \subfloat[Experiment A: MSE as a function of the number of observation points $n$.]{\label{fig:consistency:unif}%
           \includegraphics[width=17cm]{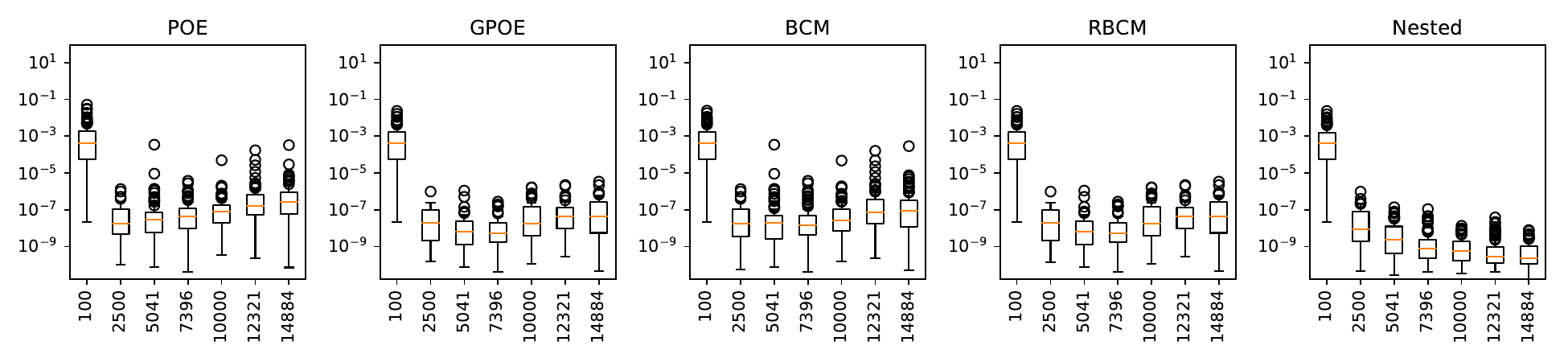}} \\
  \subfloat[Experiment B: MSE as a function of the number of observation points $n$.]{\label{fig:consistency:beta}%
           \includegraphics[width=17cm]{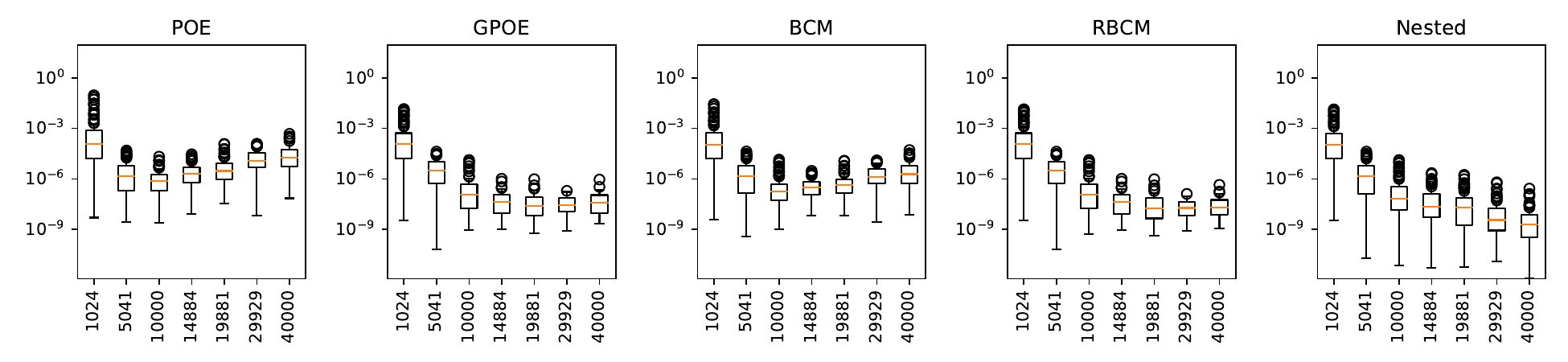}} \\
           \caption{Illustration of the (non)-consistency of the various methods discussed in this paper. (a, b): Details of the experiment settings: Samples of the test functions and distribution of the input points (histogram in the background). The vertical dashed line denotes the test point where MSE is computed. (c, d) Prediction accuracy (MSE) versus the number of observation points.}
	\label{fig:consistency}
\end{figure}

The results of the experiments are shown in panels (c) and (d) of Fig.~\ref{fig:consistency}. First of all, the non-consistency of the methods POE, GPOE, BCM and RBCM is striking: the MSE does not only fail to converge to zero but it actually increases when the number of observation points is greater than $5.\, 10^3$ (Exp. A) or $20.\, 10^4$ (Exp. B). Note that Proposition~\ref{prop:nonconsistency} only shows the existence of a training set where the variance based aggregation methods under study are non consistent: it is thus of significant practical interest to observe this behavior on these simple examples with reasonable settings. On the other hand, the \NK{} aggregation does converge toward zero as guaranteed by Proposition~\ref{prop:consistency}.


\subsection{The Gaussian process perspective}
\label{section:GP:perspective}

\nicolas{This section introduces an alternative construction of the aggregated predictor where the prior} process $\Y$ is replaced by another process $\Y_\agg$ for which $M_\agg(x)$ and $v_\agg(x)$ correspond exactly to the conditional expectation and variance of $\Y_\agg(x)$ given $\Y_\agg(X)$. As discussed in~\citet{quinonero2005},  this \nicolas{construction implies that the} proposed aggregation is not only as an approximation of the full model but also as an exact method for a slightly different prior (as illustrated in the further commented Fig.~\ref{fig:process}). This type of decomposition can naturally occur in the context of predictive processes or low-rank Kriging models, see~\cite{finley2009improving, banerjee2008gaussian,cressie2008fixed}.
As a consequence, it also provides conditional cross-covariances and samples for the aggregated models. In particular, all the methods developed in the literature based on Kriging predicted covariances, such as \citet{marrel2009} for sensitivity analysis and \citet{chevalier2013} for optimization, may hence be applied to the aggregated model in \citet{rulliere2018nested}. 
Recall that $\transpose{(\M_1,\ldots, M_p,\Y)}$ is a centered process with finite variance on the whole input space $D$. \nicolas{The $p \times 1$ cross-covariance vector is defined as } $k_M(x,x')=\Cov{M(x),Y(x')}$ and the $p \times p$ cross-covariance matrix $K_M(x,x')=\Cov{M(x),M(x')}$, for all $x, x' \in D$.
\nicolas{These definitions result in a minor notation overloading with the definitions introduced in Sect.~\ref{section:rulliere} ($K_M(x,x) = K_M(x)$ and $k_M(x,x) = k_M(x)$), but context should be sufficient to avoid confusion}.
\nicolas{The following definition introduces $\Y_\agg$ which is a Gaussian process for which $\M_\agg$ and $v_\agg$ are the conditional mean and variance of $\Y_\agg$ given $\Y_\agg(X)$:}
\begin{definition}[Aggregated process]\label{def:ProcessAgrege}
	$\Y_\agg$ \nicolas{is defined} as $\Y_\agg = \M_\agg + \varepsilon'_\agg $ where $\varepsilon'_\agg$ is an independent replicate of $\Y - \M_\agg$ and with $M_{\agg}$ as in~\eqref{eq:aggregation:rulliere}.
\end{definition}

As $ \Y = \M_\agg + (\Y - \M_\agg)$, the difference between $\Y$ and $\Y_\agg$ is that $\Y_\agg$ neglects the covariances between  $\M_\agg$ and the residual $\Y - \M_\agg$.
\begin{proposition}[Gaussian process perspective]\label{prop:PriorPosterior}
If $\M_\agg$ is a deterministic and interpolating function of $\Y(X)$, \ie if for any $x\in D$ there exists a deterministic function $g_x:\R^n \rightarrow \R$ such that $\M_\agg(x)=g_x(\Y(X))$ and if $\M_\agg(X) = \Y(X)$, or in particular under linearity and interpolation assumptions~\ref{H1} and~\ref{H2} then
	\begin{equation}
\accoladesplit{
		\M_\agg(x) & = \Esp{\Y_\agg(x) | \Y_\agg(X)} \virguleacc\\
		v_\agg(x) & = \Var{\Y_\agg(x) | \Y_\agg(X)} \pointacc
}
	\end{equation}
\end{proposition}

As already stated, given the observations $Y_\agg(X)$, the conditional process $\Y_\agg$ is interesting since its conditional mean and variance, at any point $x$, correspond to the approximated conditional mean and variance of the process $Y$ obtained by the \NK{} technique.
It is thus natural to consider sample paths of this conditional process $Y_\agg$. \nicolas{In the Gaussian setting, studying the unconditional (prior) distribution of the centred process $Y_\agg$ and the conditional (posterior) distribution of $Y_\agg$ given the observations $Y_\agg(X)$ boils down to} studying the prior and posterior covariances of $Y_\agg$. The covariance $k_\agg$ of the process $\Y_\agg$ can be calculated and shown to coincide with the one of the process $\Y$ at several locations, 
in particular, denoting $k_\agg(x,x')  = \Cov{\Y_\agg(x), \Y_\agg(x')}$, one can show that for all $x \in D$, $\Y(x)$ and $\Y_\agg(x)$ have the same variance: $k_\agg(x,x) = k(x,x)$. Furthermore, under the interpolation assumption~\ref{H2}, $k_\agg(X,X) = k(X,X)$. Figure~\ref{fig:kaggk} illustrates the difference between the covariance functions $k$ and $k_\agg$, using the same settings as in Fig.~\ref{fig:process}. \nicolas{Each panel of the figure deserves some specific comments:}
\begin{itemize}
	\item[(a)] the absolute difference between the two covariance functions $k$ and $k_\agg$ is small. Furthermore, the identity $k_\agg(X,X) = k(X,X)$ is illustrated : as $0.3$ is a component of $X$, $k_\agg(0.3,x_k) = k(0.3,x_k)$ for any of the five components $x_k$ of $X$.
	\item[(b)] the contour lines for $k_\agg$ are not straight lines, as it is the case for stationary processes. In this example, $\Y$ is stationary whereas $Y_\agg$ is not. However, the latter only departs slightly from the stationary assumption.
	\item[(c)] the difference $k_\agg - k$ vanishes at some places, among which are the places of the bullets points and the diagonal which correspond respectively to $k_\agg(X,X) = k(X,X)$ and $k_\agg(x,x) = k(x,x)$. Furthermore, the absolute differences between the two covariances functions are again quite small. It also shows that the pattern of the differences is quite complex.
\end{itemize}
\begin{figure}[htp]
  \centering
  \subfloat[Covariance functions $k_\agg(a, \cdot)$ (solid lines) and $k(a, \cdot)$ (dashed lines) with $a = 0.3 \in X$  and $a = 0.85 \notin X$.]{\label{fig:diff}\includegraphics[width=4.5cm]{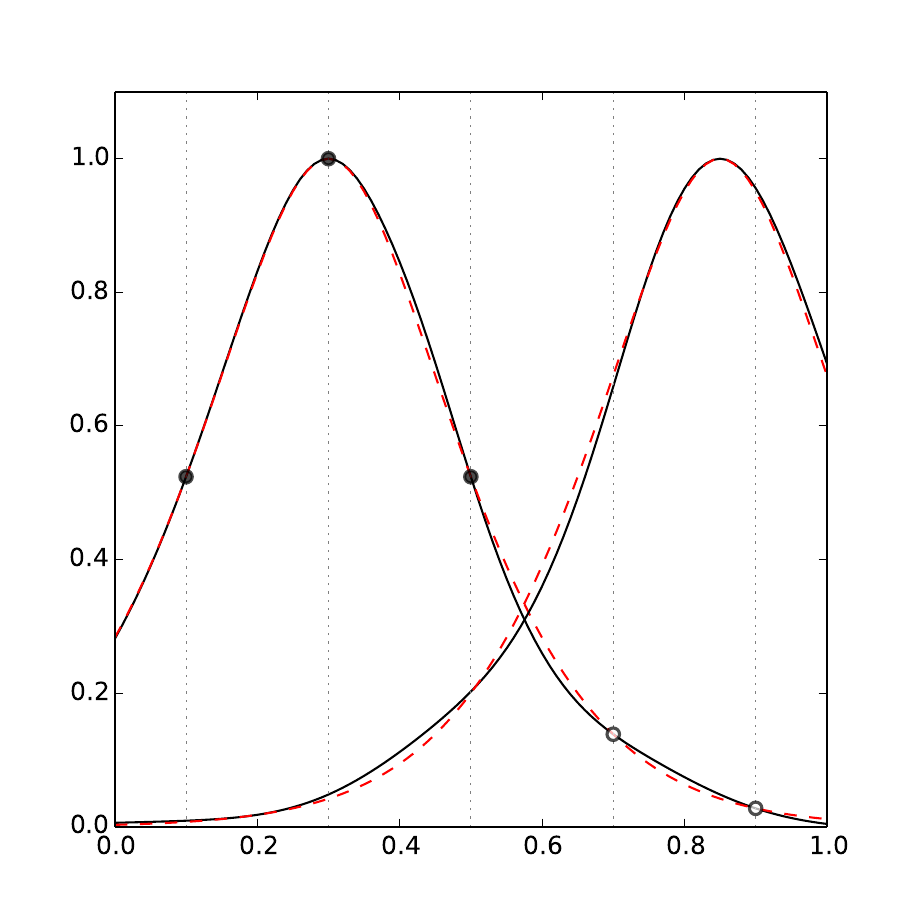}} \qquad
  \subfloat[Contour plot of the modified covariance function $k_{\agg}$.]{\label{fig:modifiedprior}\includegraphics[width=4.5cm]{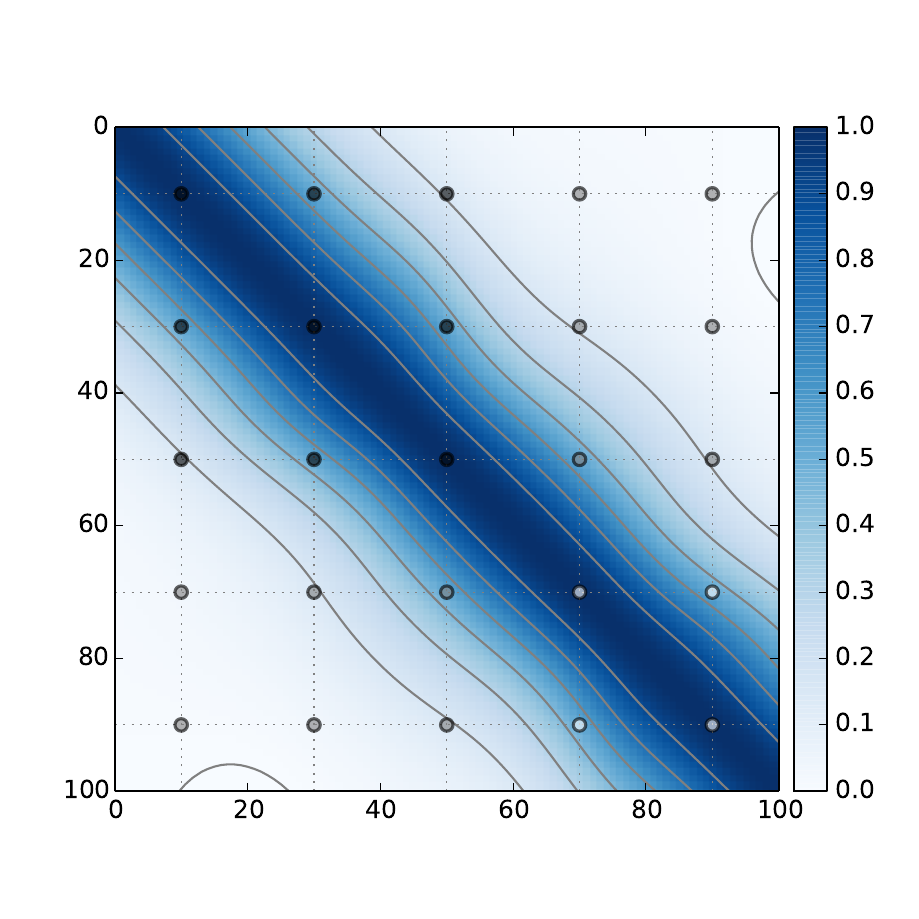}} \qquad
  \subfloat[Image plot of the difference between covariance functions $k_{\agg} - k$.]{\label{fig:diff2}\includegraphics[width=4.5cm]{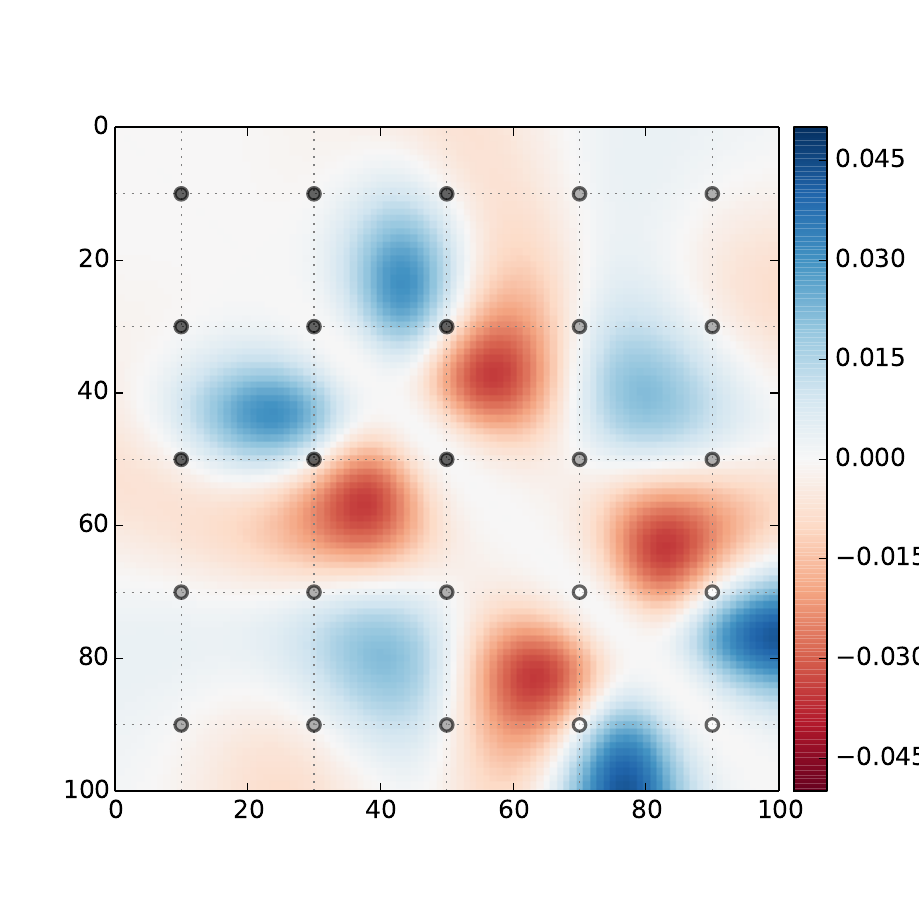}}
  \caption{Comparisons of the modified covariance $k_\agg$ and the initial covariance $k$. The horizontal and vertical dotted lines correspond to the locations of the observation points $x_i$ for $i \in \set{1,\dots,5}$. The bullets indicate locations where $k_\agg(x_i,x_j) = k(x_i,x_j)$.}
  \label{fig:kaggk}
\end{figure}

The previous considerations may help understanding the differences between $Y$ and $Y_\agg$, and thus the approximation that is made by the \NK{} technique. Another interest of $Y_\agg$ is that one can introduce conditional cross-covariances and sample paths. 
The following proposition shows that conditional (posterior) cross-covariances of $Y_\agg$ can be easily computed. In particular, it enables the computation of conditional sample paths of $Y_\agg$.
The proposition also gives some simplifications that \nicolas{make computations tractable} even in the case where the number of observations is large.

\begin{proposition}[Posterior covariances of $Y_\agg$]\label{prop:condCovYagg}
Define the conditional (posterior) cross-covariances of the process $Y_\agg$ given $Y_\agg(X)$ as
\begin{equation}
c_\agg(x,x')=\Cov{Y_\agg(x), Y_\agg(x') \sachant Y_\agg(X)} \, ,
\end{equation}
with $x, x' \in D$. Assume that  $(M, \Y )$ is Gaussian, then $\Y_\agg$ is also Gaussian and the following results hold:
\begin{enumerate}[label=(\roman*)]
\item \label{item:postCov1} The posterior covariance function $c_\agg$ writes, for all $x, x' \in D$,
\begin{equation}\label{Eq:CrossCov:Expr1}
	c_\agg(x,x')  =  k_\agg(x,x') -  k_\agg(x,X) k_\agg(X,X)^{-1} k_\agg(X,x').
\end{equation}
\item \label{item:postCov2}
Denote ${\alpha_\agg(x)}=K_M(x)^{-1}k_M(x)$. Under linear and interpolation assumptions~\ref{H1} and~\ref{H2},
\begin{eqnarray}
c_\agg(x,x') 
&=& k(x,x') - \transpose{\alpha_\agg(x)} k_M(x,x') -  k_M(x',x)\alpha_\agg(x') + \transpose{\alpha_\agg(x)} K_M(x,x') \alpha_\agg(x') \, , \label{Eq:CrossCov:Expr2}
\end{eqnarray}
\item \label{item:postCov3} Under linear and interpolation assumptions~\ref{H1} and~\ref{H2},
\begin{eqnarray}
c_\agg(x,x') 
&=&  \Esp{\left(Y(x)-M_\agg(x)\right)\left(Y(x')-M_\agg(x')\right)}\, .\label{Eq:CrossCov:Expr3}
\end{eqnarray}
In other words, conditional covariances can be understood as prior covariances between residuals.
\end{enumerate}
\end{proposition}

It should be noted that computing $c_\agg$ or generating conditional samples of $Y_\agg$ by using Eq.~\eqref{Eq:CrossCov:Expr1} requires to inverse the $n \times n$ matrix $k_{\agg}(X,X)$ which is computationally costly for large $n$. 
On the contrary, computing $c_\agg$ by using Eq.~\eqref{Eq:CrossCov:Expr2} does only require the computation of covariances between predictors and is tractable even with large datasets. Consider the prediction problem with $n$ observation points and $q$ prediction points where both $n$ and $q$ can be large, with $q=o(n)$.
Consider a reasonable dimension $d \le O(q)$ and a typical number  of sub-models $p=\sqrt{n}$. The complexity for obtaining the \NK{} mean and variance $\set{M_\agg(x), v_\agg(x)}$ for all prediction points is $\mathcal{C}=O(qn^2)$ in computational complexity and $\mathcal{S}=O(nq)$ in storage requirement for the fastest implementations~\citep[see][]{rulliere2018nested}. This storage requirement can be reduced to $\mathcal{S}=O(n)$ when recalculating some quantities. When computing $\set{M_\agg(x), v_\agg(x)}$ together with output covariances $\set{c_\agg(x,x')}$ for all prediction points, using Eq.~\eqref{Eq:CrossCov:Expr2}, one can show that the reachable computational complexity is unchanged and is $\mathcal{C}=O(qn^2)$ when $pq \le n$, or becomes $\mathcal{C}= O(q^2pn)$ otherwise. The associated storage requirement becomes $\mathcal{S}=O(nq^2)$. At last, in the more general case where $O(n^{1/2}) \le p \le O(n^{2/3})$ and $q \le O(n^{2/3})$, one can show that computational complexity is $\mathcal{C}=O(qn^2)$ without computing the covariances $c_\agg(x,x')$ or $\mathcal{C}=O(q^2p^2)$ when computing these covariances.

This Gaussian Process perspective and Proposition~\ref{prop:condCovYagg} \nicolas{can be combined to define unconditional and conditional sample paths. This is illustrated in Fig.~\ref{fig:process} which displays prior and posterior samples of a process $Y_{\agg}$ based on a process $\Y$ with squared exponential covariance $k(x,x') = \exp \left( -12.5 (x-x')^2 \right)$. In this example, the test function is $f(x)=\sin(2 \pi x) +x$, and the input $X = \transpose{(0.1, 0.3, 0.5, 0.7, 0.9)}$ is divided into $p=2$ subgroups $X_1 = \transpose{(0.1, 0.3, 0.5)}$ and $X_2 = \transpose{(0.7, 0.9)}$.}

\begin{figure}[htp]
  \centering
  \subfloat[Unconditional samples.]{\label{fig:priorsample}\includegraphics[width=7cm]{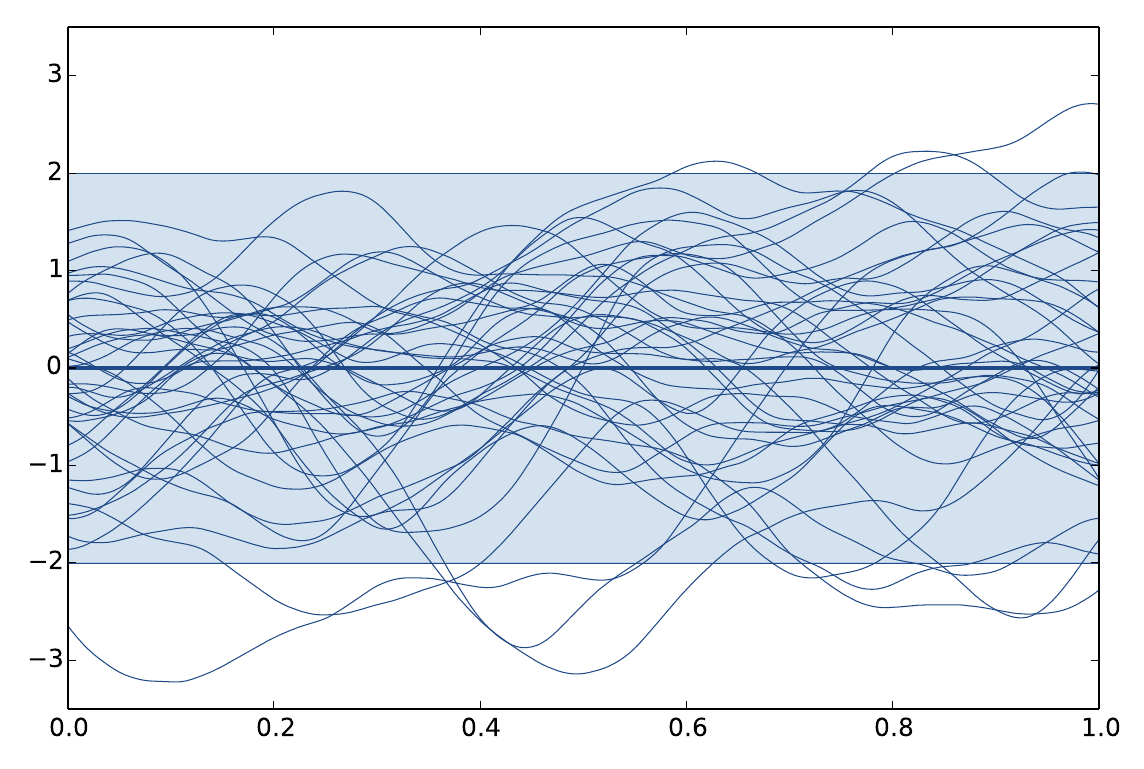}} \qquad
  \subfloat[Conditional samples.]{\label{fig:postsamples}\includegraphics[width=7cm]{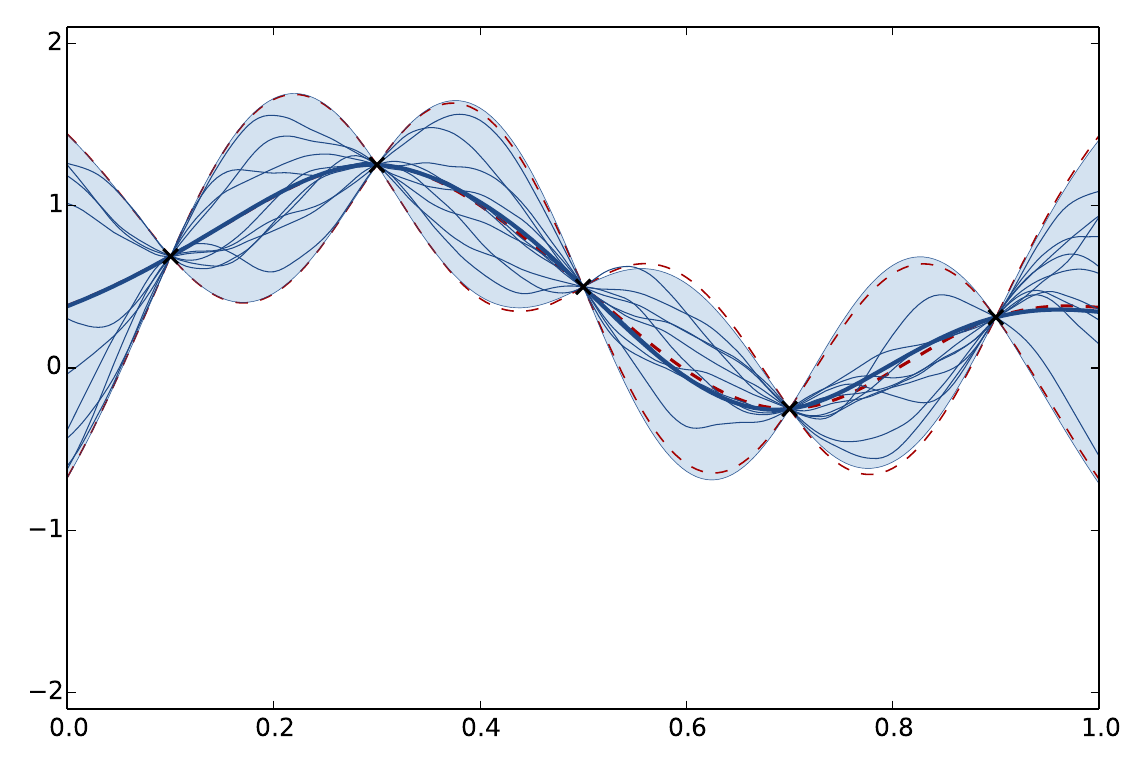}}
  \caption{Illustration of the modified process $Y_{\agg}$. (a) Unconditional sample paths from the modified Gaussian process $\Y_\agg$, with mean $0$ and covariance $k_\agg$. (b) Conditional sample paths of $\Y_\agg$ given $\Y_\agg(X) = f(X)$, with mean $M_\agg$ and covariance $c_\agg$. The thick lines and the blue areas correspond to the means and $95 \%$ confidence intervals for $\Y_\agg$. The dashed red lines are the mean and $95 \%$ confidence intervals for the full model with the original prior $Y$.}
	\label{fig:process}
\end{figure}

\nicolas{Proposition~\ref{prop:condCovYagg} can be used to draw similarities between \NK{} and low-rank Kriging \citep[see][and references therein]{stein14limitations}. In both cases, the predictions can be seen as a tractable approximation of an initial model, but they also correspond to an exact posterior for their stated covariance models (which is not stationary, in general, see Figure~\ref{fig:kaggk}). The main difference between the two methods is that contrarily to low-rank Kriging, \NK{} remains a non-parametric approach. This however comes with an additional computational cost which comes from matrix inverse that need to be computed at prediction time.}

\nicolas{Knowing that the predictor $M_{\agg}$ is a conditional expectation for the process $Y_{\agg}$ can be used} to analyze its error for predicting $Y(x)$, by studying the differences between the distributions of $Y$ and $Y_{\agg}$, in the same vein as in \citet{stein2012interpolation} or \citet{putter2001effect}.
\nicolas{The next section provides more details on the prediction errors made by choosing $\Y_\agg$ in place of $Y$ as a prior}.

\subsection{Bounds on aggregation errors}\label{sec:bounds}

This section aims at studying the differences between the aggregated model $M_{\agg},\vv_{\agg}$ and the full one $M_{\full},\vv_{\full}$. \nicolas{This section focuses} on the case where $M(x)$ is linear in $Y(X)$, \ie there exists a $p \times n$ deterministic matrix $\Lambda(x)$ such that $\M(x) = \Lambda(x) \Y(X)$. \nicolas{This results in}
	\begin{equation}
\accolade{
		\M_\agg(x)-M_\full(x) &=& -k(x,X) \Delta(x) Y(X) \virguleacc\\
	    v_\agg(x)-v_\full(x) &=& k(x,X) \Delta(x) k(X,x) \virguleacc
		}
		\label{eq:diffaggfull}
	\end{equation}
where $\Delta(x)=k(X,X)^{-1}-\transpose{\Lambda(x)}  \big(\Lambda(x) k(X,X) \transpose{\Lambda(x)} \big)^{-1} \Lambda(x)$, as soon as $\Lambda(x) k(X,X) \transpose{\Lambda(x)}$ is invertible.

\nicolas{As illustrated in Fig.~\ref{fig:kaggk}, the covariance functions $k$ and $k_\agg$ are very similar. The following proposition shows that the difference between these covariances can be linked to the aggregation error and can provide a bounds for the absolute errors}.

\begin{proposition}[Errors using covariance differences]
\label{prop:bounds}\label{prop:maxerror}
Under the linear and interpolation assumptions~\ref{H1} and~\ref{H2}, the differences between the full and aggregated models write as differences between covariance functions:
\begin{equation}
\accoladesplit{
	\Esp{(\M_\agg(x) - \M_\full(x))^2} &= \norm{k(X,x)-k_\agg(X,x)}_K^2 \virguleacc\\
	v_\agg(x) - v_\full(x) &= \norm{k(X,x)}_K^2 - \norm{k_\agg(X,x)}_K^2 \pointacc
}
\label{eq:diff}
\end{equation}
The absolute differences can be bounded:
	\begin{equation}\label{eq:boundsnorms}
		\accolade{
			\abs{\M_\agg(x)-M_\full(x)} & \le & \norm{k(X,x)-k_\agg(X,x)}_K \norm{Y(X)}_K \virguleacc\\
		\abs{v_\agg(x)-v_\full(x)} & \le & \norm{k(X,x)}_K^2 \virguleacc
		}
	\end{equation}
where $\norm{u}_K^2 = \transpose{u} k(X,X)^{-1} u$. Assuming that the smallest eigenvalue $\lambda_{\min}$ of $k(X,X)$ is non zero, this norm can be bounded by $\norm{u}_K^2 \le \frac{1}{\lambda_{\min}} \norm{u}^2$ where $\norm{u}$ denotes the Euclidean norm. Furthermore, since $v_\full(x) = \Esp{(\Y(x) - \M_\full(x))^2}$, then
	\begin{equation}
		0 \le v_\agg(x) - v_\full(x) \le \min\limits_{k \in \set{1,\ldots, p}} \Esp{(Y(x)-\M_k(x))^2} - v_\full(x) \, .
		\label{eq:va-vfullbounds}
	\end{equation}
\end{proposition}
Note that previous result is provided for a given number $n$ of observations, for a finite a dimensional $n\times n$ matrix $X$. The asymptotic of the bounds as $n$ grows to infinity depends on the design sequence and the nature of the asymptotic setting (e.g., expansion domain or fixed domain). It would require further developments that are not considered here. 

Proposition \ref{prop:bounds} implies that the \NK{} aggregation has two desirable properties that are detailed in Remarks~\ref{rem:far} and~\ref{rem:pointEquivalence} (with proofs in Appendix). 

\begin{remark}[Far prediction points]\label{rem:far}
For a given number of observations $n$ and a given design $X$, if one can choose a prediction point $x$ far enough from the observation points in $X$, in the sense $\norm{k(X,x)}\le \epsilon$ for any given $\epsilon>0$, then $|\M_\agg(x)-M_\full(x)|$ and $|v_\agg(x)-v_\full(x)|$ can be as small as desired.
\end{remark}

One consequence of the previous remark is that when the covariances between the prediction point $x$ and the observed ones $X$ become small, both models tend to predict the unconditional distribution of $\Y(x)$. This is a natural property that is desirable for any aggregation method but it is not always fulfilled. For example, aggregating two sub-models with POE leads to overconfident models with wrong variance as discussed in \citet{deisenroth2015}.

The difference between the full model and the aggregated one of Fig.~\ref{fig:process} is illustrated in Fig.~\ref{fig:ecart}. Various remarks can be made on this figure. First, the difference between the aggregated and full model is small, both on the predicted mean and variance. Second, the error tends toward 0 when the prediction point $x$ is far away from the observations $X$. This illustrates Proposition~\ref{prop:maxerror} in the case where $\norm{k(X,x)}$ is small. Third, it can be seen that the bounds on the left panel are relatively tight on this example, and that both the errors and their bounds vanish at observation points. At last, the right panel shows $v_\agg(x) \ge v_\full(x)$. This is because the estimator $\M_\agg$ is expressed as successive optimal linear combinations of $\Y(X)$, which have a quadratic error necessarily greater or equal than $\M_{full}$ which is the optimal linear combination of $\Y(X)$. Panel (b) also illustrates that the bounds given in~\eqref{eq:va-vfullbounds} are relatively loose. This means that the nested aggregation is more informative than the most accurate sub-model.
\begin{figure}
  \centering
  \subfloat[differences between predicted means $m_\agg(x) - m_\full(x)$.]{\includegraphics[width=7cm]{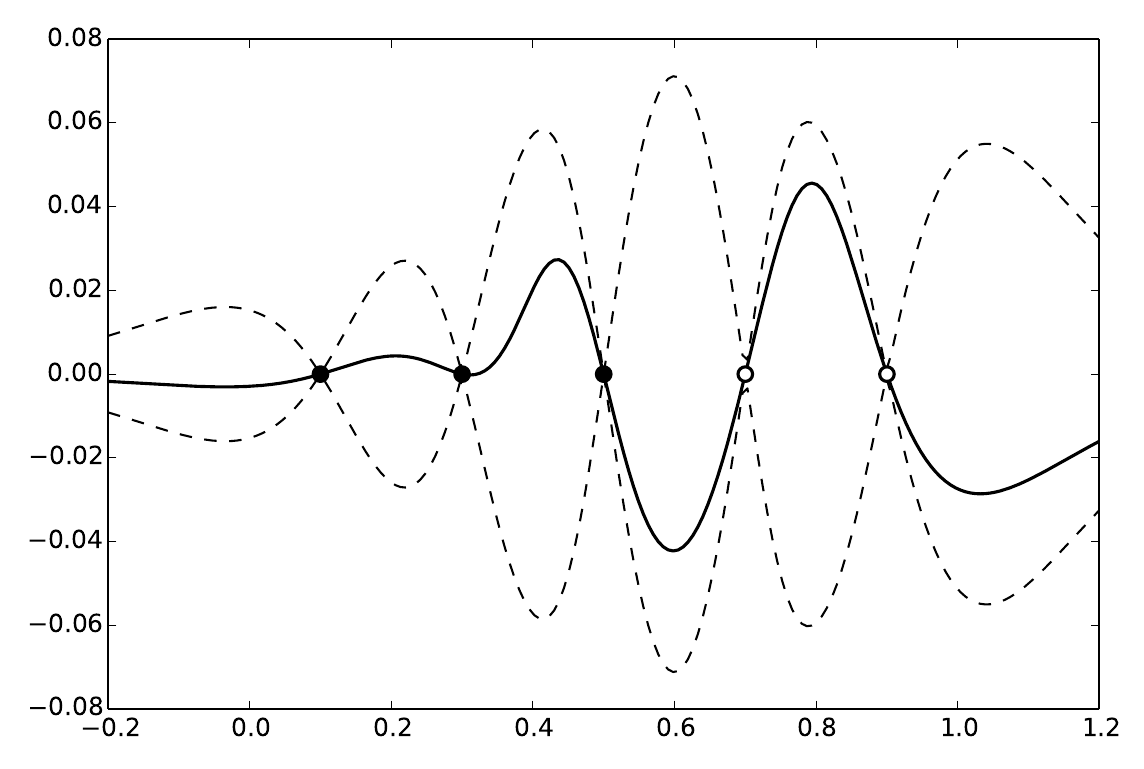}} \qquad
  \subfloat[differences between predicted variances $v_\agg(x) - v_\full(x)$.]{\includegraphics[width=7cm]{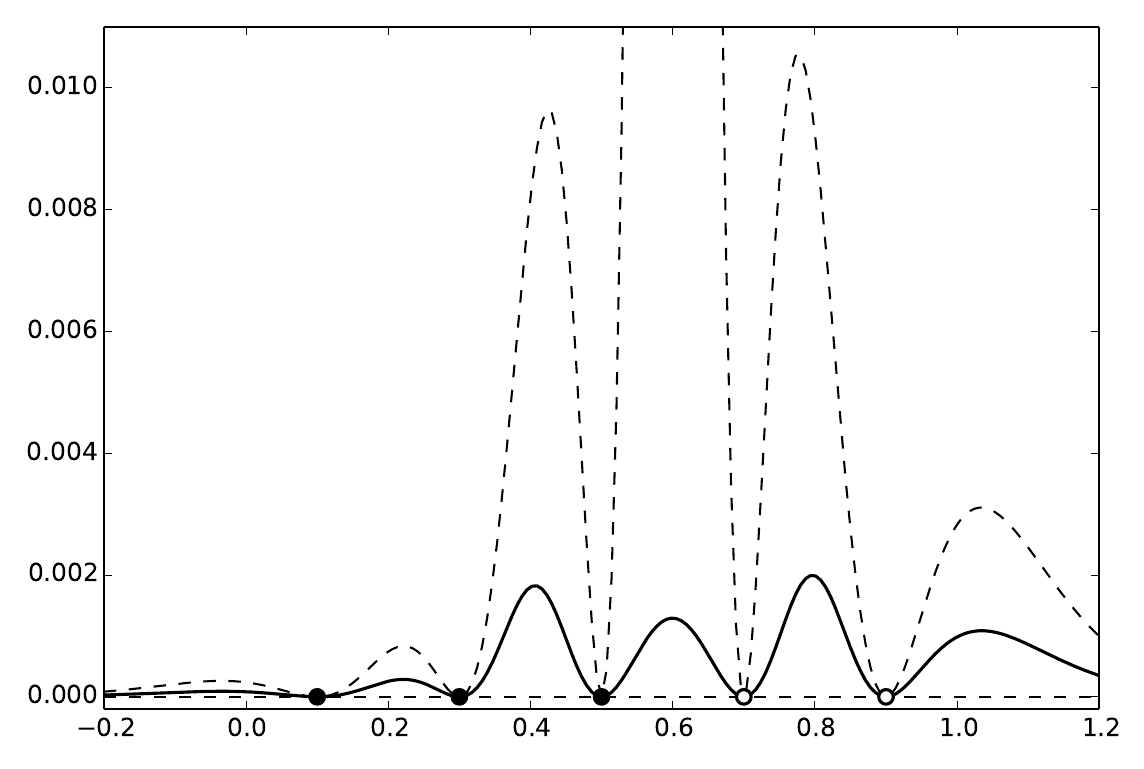}} \qquad
  \caption{Comparisons of the full and aggregated model. The dashed lines correspond to the bounds given in Proposition~\ref{prop:bounds}: $\pm \lambda_{\min}^{-1/2} \norm{k(X,x)-k_\agg(X,x)}$ on panel (a) and bounds of~\eqref{eq:va-vfullbounds} on panel (b).}
\label{fig:ecart}
\end{figure}

At last, the following remark gives another very natural optimality property that is however not satisfied by other aggregation methods such as POE, GPOE, BCM and RBCM (see Sect.~\ref{section:non:consistency}): if the sub-models contain enough information, the aggregated model corresponds to the full one. 
\begin{remark}[Fully informative sub-models]\label{rem:pointEquivalence} Assume~\ref{H1} that $\M(x)$ is linear in $\Y(X)$: $\M(x) = \Lambda(x) \Y(X)$ and that $\Lambda(x)$ is a $n\times n$ matrix with full rank, then
	\begin{equation}
	\accolade{
		\M_\agg(x) &=& M_\full(x) \virguleacc\\
		v_\agg(x) &=& v_\full(x) \pointacc
		}
		\label{eq:fullyinformative}
	\end{equation}

Furthermore, if (H3) also holds,
	\begin{equation}
		\Y_\agg \stackrel{law}{=} \Y \quad \text{ and thus } \quad \Y_\agg |  \Y_\agg(X) \stackrel{law}{=} \Y | \Y(X).
	\end{equation}
	In other words, there is no difference between the full and the approximated models when $\Lambda(x)$ is invertible. 
\end{remark}
Note that there is of course no computational interest in building and merging fully informative sub-models since it requires computing and inverting a matrix that has the same size as $k(X,X)$ so there is no complexity gain compared to the full model.

\section{Analysis of the impact of the group choice} \label{section:impact:clustering}

\nicolas{This section studies} the impact of the choice of the partition $X_1,\ldots,X_p$ of a set of $n$ two-by-two distinct observation points $\{x_1,\ldots,x_n\}$, on the quality of the predictor obtained by aggregating $p$ Gaussian process models based on $X_1,\ldots,X_p$. 

\subsection{Theoretical results in dimension 1}

\nicolas{This section focuses on the univariate case $d=1$ where  with the input locations $x_1,\ldots,x_n \in \R$ are fixed and distinct points and where $Y$ is} a centered Gaussian process with exponential covariance function $k$ defined by
\begin{equation} \label{eq:kexp}
k(t_1,t_2) = \sigma^2 \exp( - |t_1 - t_2| / \theta) , \quad t_1,t_2 \in \R
\end{equation}%
for fixed $(\sigma^2,\theta) \in (0,\infty)^2$. \nicolas{This choice of covariance function makes $Y$ a Markovian GP \citep{Ying91}, which will prove useful to derive theoretical properties on the influence of clustering. More precisely, the idea is to assess} whether selecting the groups $X_1,\ldots,X_p$ based on distances (i.e., placing observation points close to each other in the same group) is beneficial \nicolas{for the approximation accuracy or not. In dimension $1$, the concept of perfect clustering can be defined} as follows.

\begin{definition}[Perfect clustering] \label{def:perfect:clustering}
A partition $X_1,\ldots,X_p$ of $\{x_1,\ldots,x_n\}$, composed of non-empty groups, is a perfect clustering if there does not exist any triplet $u,v,w$, with $u,v \in X_i$ and $w \in X_j$ with $i,j \in \{1 \ldots,p\}$, $i \neq j$, and so that $ u < w <v $.
\end{definition}

The above definition means that the groups $X_1,\ldots,X_p$ are constituted of consecutive points. A partition $X_1,\ldots,X_p$ is a perfect clustering if and only if it can be reordered as $X_{i_1},\ldots,X_{i_p}$ with $\{i_1,\ldots,i_p\} = \{ 1 ,\ldots ,p\}$ and so that for any $u_{1} \in X_{i_1},\ldots,u_p \in X_{i_p}$, the $u_i$ are ordered $u_1 < \ldots < u_p$.

\nicolas{The next proposition shows} that the \NK{} predictor coincides with the predictor based on the full Gaussian process model, if and only if $X_1,\ldots,X_p$ is a perfect clustering. \nicolas{It thus provides a theoretical confirmation that placing observations points close to each other in the same group is beneficial to} the \NK{} procedure.

\begin{proposition}[\MAJNK{} and perfect clustering] \label{prop:perfect:clustering:markov}
Consider an exponential covariance function $k$ in dimension $d=1$, as in~\eqref{eq:kexp}. Let $\M_\full(x)$ be the full predictor as in~\eqref{eq:fullmodel} and let $M_{\agg}(x) $ be the \NK{} predictor as in~\eqref{eq:aggregation:rulliere}, where $M_1,\ldots,M_p$ are the Gaussian process predictors based on the individual groups $X_1,\ldots,X_p$, as in Sect.~\ref{section:non:consistency}, that is assumed to be non-empty. 
Then, $\mathbb{E}[ ( \M_\full(x) - \M_{\agg}(x) )^2] = 0$ for all $x \in \R$ if and only if $X_1,\ldots,X_p$ is a perfect clustering.
\end{proposition}

\begin{proof}
Let $(x,v_1,\ldots,v_r)$ be $r+1$ two-by-two distinct real numbers. If $x < \min(v_1,\ldots,v_r)$, then the conditional expectation of $Y(x)$ given $Y(v_1),\ldots,Y(v_r)$ is equal to $\exp( - | min(v_1,\ldots,v_r) - x | / \theta  ) Y( min(v_1,\ldots,v_r) )$ \citep{Ying91}. Similarly, if $x > \max(v_1,\ldots,v_r)$, then the conditional expectation of $Y(x)$ given $Y(v_1),\ldots,Y(v_r)$ is equal to $\exp( - | max(v_1,\ldots,v_r) - x | / \theta  ) Y( max(v_1,\ldots,v_r) )$. If $\min(v_1,\ldots,v_r) <  x < \max(v_1,\ldots,v_r)$, then the conditional expectation of $Y(x)$ given $Y(v_1),\ldots,Y(v_r)$ is equal to $a Y( x_{<} ) + bY( x_{>} )$ where $x_{<}$ and $x_{>}$ are the left-most and right-most neighbors of $x$ in $\{ v_1 , \ldots , v_r \}$ and where $a,b$ are non-zero real numbers \citep{bachoc2017cross}. Finally, because the covariance matrix of $Y(v_1),\ldots,Y(v_r)$ is invertible, two linear combinations $\sum_{i=1}^r a_i Y(v_i)$ and $\sum_{i=1}^r b_i Y(v_i)$ are equal almost surely if and only if $(a_1,\ldots,a_r) = (b_1,\ldots,b_r)$.

Assume that $X_1,\ldots,X_p$ is a perfect clustering and let $x \in \R$. \nicolas{It is known from \citet{rulliere2018nested} that} $M_\agg(x) = M_\full(x)$ almost surely if $x \in \{x_1,\ldots,x_n\}$. Consider now that  $x \not \in \{x_1,\ldots,x_n\}$.

If $x < \min(x_1,\ldots,x_n)$, then for $i=1,\ldots,p$, $M_i(x) = \exp( - | x_{j_i} - x | / \theta  ) Y(x_{j_i})$ with $x_{j_i} = \min \{ x ; x \in X_i \}$. Let $i^* \in \{1,\ldots,p\}$ be so that $\min(x_1,\ldots,x_n) \in X_{i^*}$. Then $\M_\full(x) = \exp( - | x_{j_{i^*}} - x | / \theta  ) Y(x_{j_{i^*}})$. As a consequence, the linear combination $\lambda_x^t M(x) $ minimizing $\mathbb{E}[ (\lambda^t M(x) - Y(x) )^2]$ over $\lambda \in \R^p$ is given by $\lambda_x = e_{i^*}$ with $e_r$ the $r$-th base column vector of $\R^p$. This implies that $\M_\full(x) = \M_{\agg}(x)$ almost surely. Similarly, if $x > \max(x_1,\ldots,x_n)$, then $\M_\full(x) = \M_{\agg}(x)$ almost surely.

Consider now that there exists $u \in X_{i}$ and $v \in X_{j}$ so that $ u < x < v $ and $(u,v)$ does not intersect with $\{x_1,\ldots,x_n\}$.
 If $i = j$, then $M_i(x) = M_{\full}(x)$ almost surely because the left-most and right-most neighbors of $x$ are both in $X_i$.
Hence, also $M_{\agg}(x) =M_\full(x)$ almost surely in this case. 
  If $i \neq j$, then $u = \max \{t ; t \in X_i\}$ and $v = \min \{t ; t \in X_j\}$ because $X_1,\ldots,X_p$ is a perfect clustering.
Hence, $M_i(x) = \exp( - |x - u|) Y(u)$, $M_j(x) = \exp( - |x - v|) Y(v)$ and $M_{\full}(x) = a Y(u) + bY(v)$ with $a,b \in \R$. Hence, there exists a linear combination $\lambda_i M_i(x) + \lambda_j M_j(x)$ that equals $M_{\full}(x)$ almost surely. As a consequence, the linear combination $\lambda_x^t M(x) $ minimizing $\mathbb{E}[ (\lambda^t M(x) - Y(x) ]^2)$ over $\lambda \in \R^p$ is given by $\lambda_x = \lambda_i e_{i} + \lambda_j e_j$.
Hence $\M_\full(x) = \M_{\agg}(x)$ almost surely. All the possible sub-cases have now been treated, which proves the first implication of the proposition.

Assume now that $X_1,\ldots,X_p$ is not a perfect clustering. Then there exists a triplet $u,v,w$, with $u,v \in X_i$ and $w \in X_j$ with $i,j=1 \ldots,p$, $i \neq j$, and so that $ u < w <v $. Without loss of generality \nicolas{it can further be assumed that} there does not exits $z \in X_i$ satisfying $u < z <v$. 

Let $x$ satisfies $ u < x < w $ and so that $(u,x)$ does not intersect $\{ x_1,\ldots,x_n \}$. Then $\M_{\full} (x) = a Y(u) + b Y(z)$ with $a , b \in \R \backslash \{0\}$ and $z \in \{ x_1,\ldots,x_n \}$, $z \neq v$. Also, $M_i(x) = \alpha Y(u) + \beta Y(v)$ with $\alpha , \beta \in \R \backslash \{0\}$. As a consequence, there can not exist a linear combination $\lambda^t M(x)$ with $\lambda \in \R^p$ so that $\lambda^t M(x) = a Y(u) + b Y(w)$. Indeed a linear combination $\lambda^t M(x)$ is a linear combination of $Y(x_1),\ldots,Y(x_n)$ where the coefficients for $Y(u)$ and $Y(v)$ are $\lambda_i \alpha $ and $\lambda_i \beta$, which are either simultaneously zero or simultaneously non-zero.
  Hence, $M_\agg(x)$ is not equal to $M_\full(x)$ almost surely. This concludes the proof. 
  \qedperso
\end{proof}

\nicolas{The next proposition shows} that the aggregation techniques that ignore the covariances between sub-models can never recover the full Gaussian process predictor, even in the case of a perfect clustering. This again highlights the additional quality guarantees brought by the \NK{} procedure.

\begin{proposition}[Non-perfect other aggregation methods] \label{prop:non:perfect:aggregation:markov}
Consider an exponential covariance function $k$ in dimension $d=1$, as in~\eqref{eq:kexp}. Let $p \geq 3$ and let $X_1,\ldots,X_p$ be non-empty.
Let $M_{\agg}$ be a covariance-free aggregation method  defined as in~\eqref{eq:form:aggregation:variance}, with $\alpha_k(v_1,\ldots,v_p,v) \in \R \setminus \set{0}$ for $v_1,\ldots,v_p,v \in (0,\infty)$ and $v_1 <v,\ldots,v_p < v$. Then, for all $x \in \R \setminus X$,  $\mathbb{E}[ ( \M_\full(x) - \M_{\agg}(x) )^2] > 0$.
\end{proposition}
\begin{proof}
Let $x \in \R\setminus X$. For $i=1, \ldots, p$, $0 < \vv_i(x) < \vv_{prior}(x)$, so that $\alpha_i(v_1(x),\ldots,v_p(x),v_{prior}(x)) \in \R \setminus \set{0}$.
Hence, the linear combination $M_\agg(x) = \sum_{k=1}^{p} \alpha_{k}(\vv_1(x),...,\vv_{p}(x),\vv_{prior}(x)) M_k(x)$ is a linear combination of $Y(x_1),\ldots,Y(x_n)$ with at least $p$ non-zero coefficients (since each $M_k(x)$ is a linear combination of one or two elements of $Y(x_1),\ldots,Y(x_n)$, all these elements being two-by-two distinct, see the beginning of the proof of Proposition~\ref{prop:perfect:clustering:markov}). Hence, because the covariance matrix of $Y(x_1),\ldots,Y(x_n)$ is invertible, $M_\agg(x)$ can not be equal to $M_{\full}(x)$ almost surely, since $M_{\full}(x)$ is a linear combination of $Y(x_1),\ldots,Y(x_n)$ with one or two non-zero coefficients.
  \qedperso
\end{proof}

The above Proposition applies to the POE, GPOE, BCM and RBCM procedures presented in Sect.~\ref{section:non:consistency}.

\subsection{Empirical results}
The aim of this section is to illustrate \did{Proposition}~\ref{prop:perfect:clustering:markov}, and to study the influence of the allocation of the observation points to the sub-models. Two opposite strategies are indeed possible: the first one consists in allocating all the points in one region of the input space to the same sub-model (which is then accurate in this region but not informative elsewhere). The second is to have, for each sub-model, points that are uniformly distributed among the set of observations which leads to having a lot of sub-models that are weekly informative. \nicolas{This section illustrates the impact of this choice on the \NK{} MSE}.

\nicolas{The experiment is as follow. A set of $32^2 = 1024$ observation points are distributed on a regular grid in one dimension and two methods are considered for creating 32 subsets of points: a k-means clustering and the optimal clustering which consists in grouping together sets of 32 consecutive points. These initial grouping of points can be used to build sub-models that are experts in their region of the space. In order to study the influence of the quality of the clustering is, the clusters are perturbed by looping over all observations points and for each of them the group is swapped with another random observation point with probability $p$. The value $p$ can then be used as a measure of the disorder in the group assignment: for $p=0$ the groups are perfect clusters and for $p=1$, each observation is assigned a group at random.}
   
Figure~\ref{fig:clustering_effect} (top) shows the MSE of one \did{dimensional} \NK{} models as a function of $p$, for test functions that correspond to samples of Gaussian processes and a test set of 200 uniformly distributed points. The covariance functions of the Gaussian processes are either the exponential or the Gaussian (i.e. squared exponential) kernels, with unit variance and a lengthscale such that the covariance between two neighbour points is 0.5. As predicted by \did{Proposition}~\ref{prop:perfect:clustering:markov}, the error is null for $p=0$ (which corresponds to a perfect clustering) when using an exponential kernel. Although this is not supported by theoretical guaranties, one can see that the prediction error is also the smallest for at $p=0$ for a Gaussian kernel. Finally, one can note that the choice of the initial clustering method does not have a strong influence on the MSE. \nicolas{This can probably be explained by the good performance of the k-means algorithm in one dimension}.

For the sake of completeness, the experiment \nicolas{is repeated with the same settings as above except for the input space dimension that is changed from one to five, and the locations of the $1024=4^5$ input points that are now given by the grid $\{1/8,\ 3/8,\ 5/8,\ 7/8\}^5$. With such settings, the optimal clustering of the observations can be obtained analytically with the $32=2^5$ cluster centers located at $\{1/4,\ 3/4 \}^5$. As previously, the models that perform the best} are obtained with $p=0$. Furthermore, the difference between the two clustering methods is now more pronounced and the MSE obtained with k-means is always higher than the one with the optimal clustering.
    
These two experiments, together with the theoretical result in dimension one, suggest it is good practice to apply a clustering algorithm to decide how to assign the observation points to the sub-models.

\begin{figure}
  \centering
  \subfloat[input dimension = 1, exponential kernel]{\includegraphics[width=7cm]{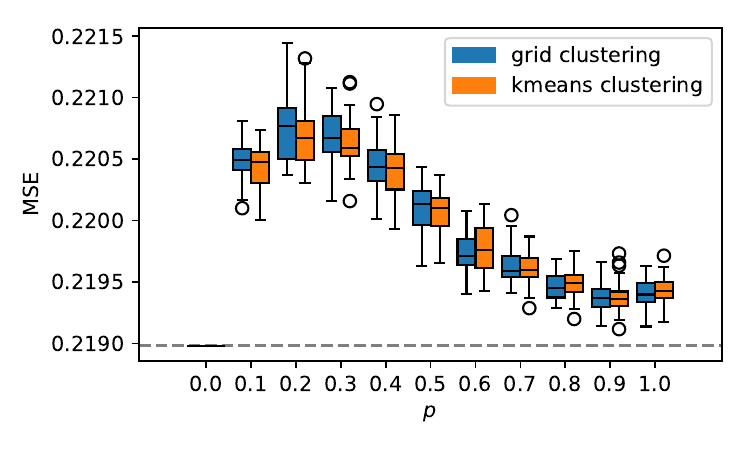}} \qquad
  \subfloat[input dimension = 1, Gaussian kernel]{\includegraphics[width=7cm]{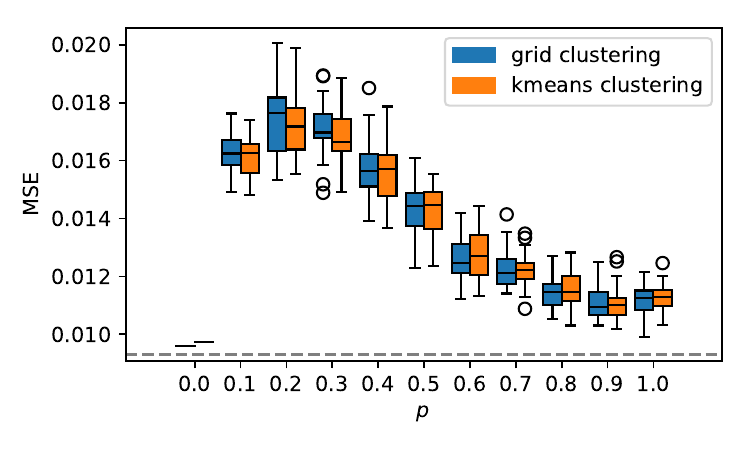}} \\
  \subfloat[input dimension = 5, exponential kernel]{\includegraphics[width=7cm]{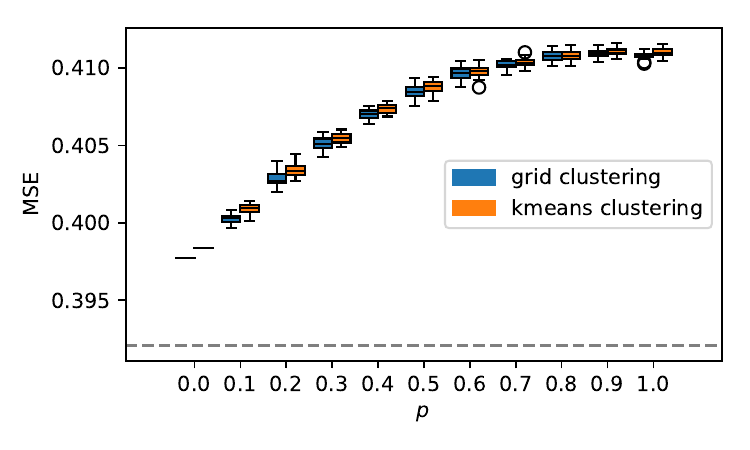}} \qquad
  \subfloat[input dimension = 5, Gaussian kernel]{\includegraphics[width=7cm]{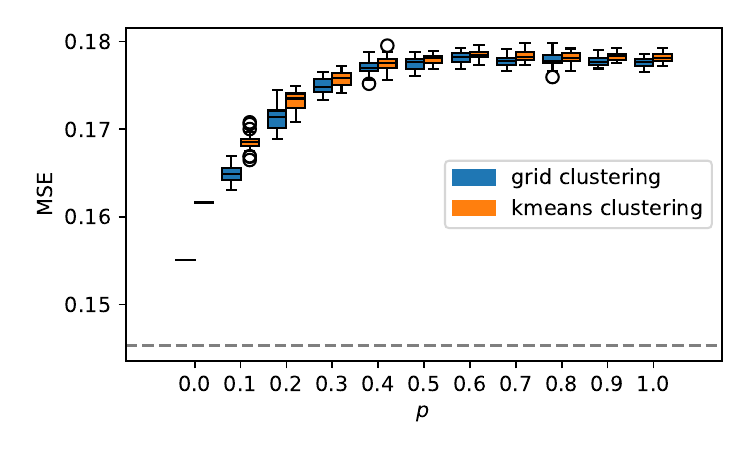}}
  \caption{\MAJNK{} MSE as a function of how clustered the inputs of the sub-models are. For small values of $p$, sub-models are built on points that tend to form clusters whereas they do not for large values of $p$. The horizontal dashed line corresponds to the optimal MSE, which is obtained with full Kriging model.}
\label{fig:clustering_effect}
\end{figure}


\section{{Extensions of the \NK{} prediction}}
\label{section:extensions}
\nicolas{This section extends \NK{} to the cases where the Gaussian process $Y$ is observed with measurement errors, and where $Y$ has a parametric mean function. It also provides theoretical guaranties similar to Propositions~\ref{prop:non:consistency} and ~\ref{prop:consistency} on the (non-)consistency of various aggregation methods in the noisy setting.}

\subsection{\fb{Measurement errors}} \label{subsection:measurement:errors}


\nicolas{This section assumes that the vector of observation is given by $Y(X) + \xi_X$ where the vector of measurement errors $\xi_X = (\xi_{1},\ldots,\xi_{n})^t$ has independent components, independent of $Y$, and where $\xi_{i} \sim \mathcal{N}(0,\eta_i)$ for $i = 1,\ldots,n$, with the error variances $\eta_1 >0,\ldots,\eta_n>0$.}
\fb{
Consider a partition $X_1,\ldots,X_p$ of $X$, where $X_1,\ldots,X_p$ have cardinalities $n_1,\ldots,n_p$. For $i=1,\ldots,p$, write $\xi_{X_i}$ as the subvector of $\xi_X$ corresponding to the submatrix $X_i$. Write also $D_i = \Cov{ \xi_{X_i} }$. Then, for $x \in D$,
the Kriging sub-model based on the noisy observations of $Y(X_i)$ is
\begin{equation} \label{eq:submodel:noisy}
M_{\eta,i}(x)
=
k(x,X_i) (k(X_i,X_i) + D_i )^{-1}
(Y(X_i) + \xi_{X_i}).
\end{equation}
Note that $M_{\eta,i}(x)$ is the best linear unbiased predictor of $Y(x)$ from $Y(X_i) + \xi_{X_i}$.
Then, the best linear unbiased predictor of $Y(x)$ from $M_{\eta,1}(x),\ldots,M_{\eta,p}(x)$ is
\begin{equation} \label{eq:M:agg:eta}
M_{\agg,\eta}(x)
=
k_{M,\eta}(x)^t K_{M,\eta}(x)^{-1} M_{\eta}(x).
\end{equation}
In \eqref{eq:M:agg:eta}, $M_{\eta}(x) = (M_{\eta,1}(x),\ldots,M_{\eta,p}(x))^t$ is the vector of sub-models, $K_{M,\eta}(x)$ is the covariance matrix of $M_{\eta}(x)$ and $k_{M,\eta}(x)$ is the $p \times 1$ covariance vector between $M_{\eta}(x)$ and $Y(x)$. \nicolas{For $i,j =1,\ldots,p$, their entries are}
\begin{align}
(k_{M,\eta}(x))_i &= k(x,X_i) (k(X_i,X_i) + D_i )^{-1} k(X_i,x) \virguleacc \\
(K_{M,\eta}(x))_{i,j} &= k(x,X_i) (k(X_i,X_i) + D_i )^{-1} k(X_i,X_j) (k(X_j,X_j) + D_j )^{-1} k(X_j,x) \pointacc
\end{align}
\nicolas{The mean square error can also be computed analytically}
\begin{equation} \label{eq:v:agg:eta}
\vv_{\agg,\eta}(x) = \Esp{ (Y(x) - M_{\agg,\eta}(x))^2 } =  k(x,x) -  \transpose{k_{M,\eta}(x)}K_{M,\eta}(x)^{-1}k_{M,\eta}(x).
\end{equation} 

\nicolas{Equations \eqref{eq:M:agg:eta} and \eqref{eq:v:agg:eta} follow from the same standard proof as in the case where there are no measurement errors, see for instance \cite{rulliere2018nested}. The computational complexity and storage requirement of these expressions are the same as their counterpart without measurement errors. In order to analyse more finely the cost of computing $M_{\agg,\eta}(x)$ and $\vv_{\agg,\eta}(x)$, the computations is broken down in four steps: (1) computing and storing the vectors $(k(x,X_i) (k(X_i,X_i) + D_i )^{-1})_{i=1,\ldots,p}$, (2) computing and storing $(M_{\eta,i}(x))_{i=1,\ldots,p}$, (3) computing and storing $K_{M,\eta}(x)$ and $(K_{M,\eta}(x))^{-1}$ and (4) computing} $M_{\agg,\eta}(x)$ and $\vv_{\agg,\eta}(x)$.

Assume that $X_1,\ldots,X_p$ have cardinalities of order $n/p$ for simplicity. Then the computational complexity of steps (1-4) are respectively $O(p (n/p)^3)$, $O(p (n/p) )$, $O( p ^2 (n/p) ^2 + p^3)$ and $O(p (n/p) + p^2)$. The total computational cost is thus $O( n^3/p^2 + n^2 + p^3 )$, which boils down to $O(n^2)$ by taking $p$ of order $n^\beta$ with $\beta \in [1/2,2/3]$ (as opposed to $O(n^3)$ for the full Kriging predictor). The storage requirement (including intermediary storage) of steps (1) to (4) is $O((n/p)^2 + p(n/p) + p^2)$. This cost becomes minimal for $p$ of order $n^{1/2}$, reaching $O(n)$ (as opposed to $O(n^2)$ for the full Kriging predictor).

\nicolas{The remaining of this section focuses on consistency results when observations are corrupted with noise.} Similarly to Prop. \ref{prop:consistency}, the following results considers infill asymptotics and a triangular array of observation points.

\begin{proposition}[Sufficient condition for \NK{} consistency with measurement errors]
Let $D$ be a fixed nonempty subset of $\mathds{R}^d$. Let $\Y$ be a Gaussian process on $D$ with mean zero and continuous covariance function $k$. Let $(x_{ni})_{1 \leq i \leq n, n \in \mathds{N}}$ be a triangular array of observation points such that $x_{ni} \in D$ for all $1 \leq i \leq n, n \in \mathds{N}$.
For \did{$n \in \mathbb{N}$}, let $X = (x_{n1},\ldots,x_{nn})^t$
and let $Y(X) + \xi_X$ be observed, where $\xi_X = (\xi_{1} , \ldots , \xi_{n})^t$ has independent components, with $\xi_{i} \sim \mathcal{N}(0,\eta_i)$, where $(\eta_i)_{i \in \mathbb{N}}$ is a bounded sequence. Let also $\xi_X$ be independent of $Y$. 
Let $x \in D$ be fixed. For $n \in \mathds{N}$, let $\M_{\eta,1}(x),...,\M_{\did{\eta},p_n}(x)$ be defined from \eqref{eq:submodel:noisy}, for a partition  $X_1,\ \dots, \ X_{p_n}$ of $X$.

Assume the following sufficient condition: for all $\epsilon >0$, there exists a sequence $(i_n)_{n \in \mathbb{N}}$, such that $i_n \in \{1 , \ldots,p_n \}$ for $n \in \mathbb{N}$ and such that the number of points in $X_{i_n}$ at Euclidean distance less than $\epsilon$ from $x$ goes to infinity as $n \to \infty$.

Then \nicolas{for} $M_{\agg,\eta}(x)$ \nicolas{defined} as in \eqref{eq:M:agg:eta},
\begin{equation*}
 \mathbb{E} \left[
\left(
Y(x) - M_{\agg,\eta}(x)
\right)^2
\right]
\to_{n \to \infty} 0.
\end{equation*}

\label{prop:consistency:noisy}
\end{proposition}

In Proposition \ref{prop:consistency:noisy}, the interpretation of the sufficient condition for consistency is that, when $n$ is large, at least one of the subsets $X_1,\ldots,X_{p_n}$ contains a large number of observation points close to the prediction point $x$.
If the minimal size of the subsets $X_1,\ldots,X_{p_n}$ goes to infinity, and if these subsets are obtained from a clustering algorithm, that is the points in a subset are close to each other, then the sufficient condition in Proposition \ref{prop:consistency:noisy} typically holds. \nicolas{This can be seen as an additional argument} in favor of selecting the subsets from a clustering algorithm. This is in agreement with \fb{Sect.} \ref{section:impact:clustering}, which conclusions also support clustering algorithms.

A particular case where the condition of Proposition \ref{prop:consistency:noisy} always holds (regardless of how the partition into subsets is made) is when the triangular array of observation points is a sequence of randomly sampled points, with a strictly positive sampling density, and when the number of subsets is asymptotically smaller than $n$.

\begin{lemma} \label{lemma:consistency:noisy:random:points}
Let $D$ be fixed, bounded with non-empty interior. Let $x$ in the interior of $D$ be fixed.
Consider a triangular array of observation points $(x_{ni})_{1 \leq i \leq n, n \in \mathds{N}}$ that is obtained from a sequence $(x_i)_{i \in \mathbb{N}}$, that is $x_{ni} = x_i$ for $1 \leq i \leq n, n \in \mathds{N}$. Assume that $(x_i)_{i \in \mathbb{N}}$ are independently sampled from a distribution with strictly positive density $g$ on $D$. Consider any sequence of partitions $(X_1,\ldots,X_{p_n})_{n \in \mathbb{N}}$ of $x_1,\ldots,x_n$. Assume that $p_n = o (n)$ as $n \to \infty$. Then, almost surely with respect to the randomness of $(x_i)_{i \in \mathbb{N}}$, the sufficient condition of Proposition \ref{prop:consistency:noisy} holds.
\end{lemma}

The theoretical setting of Lemma \ref{lemma:consistency:noisy:random:points} is realistic with respect to situations where the observation points are not too irregulalry spaced over $D$. The setting $p_n = o(n)$ is \nicolas{particularly relevant for} the \NK{} predictor, since this setting is necessary to obtain a smaller order of computational complexity than the full Kriging predictor.


\nicolas{The following Proposition shows that there are situations with measurement errors where the \NK{} predictor is consistent whereas other aggregation methods that do not use the covariances between the sub-models are inconsistent}. These situations are constructed similarly as in Proposition \ref{prop:non:consistency}. In particular, Proposition \ref{prop:non:consistency:noisy} applies to the extensions of POE, GPOE, BCM and RBCM methods to the case of measurement errors (see the references given in \fb{Sect.} \ref{section:non:consistency}). 

\begin{proposition}[non-consistency of some covariance-free aggregations with measurement errors] \label{prop:non:consistency:noisy}
\nicolas{Consider} $D$, $Y$ and $k$ satisfying the same conditions as in Proposition \ref{prop:non:consistency}.  Let $(\eta_i)_{i \in \mathbb{N}}$ be a bounded sequence. For any triangular array of observation points $(x_{ni})_{1 \leq i \leq n, n \in \mathbb{N}}$, let, for $n \in \mathbb{N}$, $X$ be the $n \times d$ matrix with row $i$ equal to $x_{n,i}^t$.
Let $\xi_X$ be as in Proposition \ref{prop:consistency:noisy}.
Then, for any partition $X_1,\ldots,X_{p_n}$ of $X$, for $i=1,\ldots,p_n$, let $n_i$ be the cardinality of $X_i$, let $M_{\eta,i}(x)$ be defined as in \eqref{eq:submodel:noisy} and $v_{\eta,i}(x) = k(x,x) - k(x,X_i)(k(X_i,X_i) + D_i)
^{-1} k(X_i,x)$, with $D_i$ also as in \eqref{eq:submodel:noisy}. Let then $M_{\agg,\eta,n}(x)$ be defined as $M_{\agg,n}(x)$ in Proposition \ref{prop:non:consistency}, with the same assumption \eqref{eq:assumption:aggregation:for:inconsistency}, with $v_i(x)$ replaced by $v_{\eta,i}(x)$. 

Then there exist a fixed $x \in D$, a triangular array of observation points $(x_{ni})_{1 \leq i \leq n, n \in \mathbb{N}}$, and a sequence of partitions $X_1,\ldots,X_{p_n}$ of $X$, that satisfy the sufficient condition of Proposition \ref{prop:consistency:noisy}, and such that
\[
\liminf_{n \to \infty}
\mathbb{E}
\left[
\left(
M_{\agg,\eta,n}(x)
-
Y(x)
\right)^2
\right]
> 0.
\]
\end{proposition}
}

\subsection{\fb{Universal Kriging}} \label{subsection:universal:kriging}

\fb{
Consider here the case where \nicolas{the Gaussian process $Z$ defined with a trend}, for $x \in D$, 
\[
Z(x) = \sum_{i=1}^m \beta_i h_i(x) + Y(x),
\]
where $Y$ is, as above, a centered Gaussian process on $D$ with mean zero and covariance function $k$, where the functions $h_1,\ldots,h_m : D \to \mathbb{R}$ are known and where the vector $\beta = (\beta_1,\ldots,\beta_m)^t$ is unknown. This is the setting of universal Kriging \citep{chiles2009geostatistics}. 

Consider the partition $X_1,\ldots,X_p$ of $X$ with cardinalities $n_1,\ldots,n_p$. For $i=1,\ldots,p$, let $H_i$ be the $n_i \times m$ matrix $( h_1(X_i),\ldots,h_m(X_i) )$.
Then the best linear unbiased predictor of $Z(x)$ given $Z(X_i)$ is \citep{sacks89design}
\[
M_{\text{UK},i}(x)
=
h(x)^t \hat{\beta}_i
+
k(x,X_i) k(X_i,X_i)^{-1} \left( Z(X_i) - H_i \hat{\beta_i} \right),
\]
with $h(x) = (h_1(x) , \ldots , h_m(x))^t$ and
\[
\hat{\beta}_i
=
\left( H_i^t k(X_i,X_i)^{-1} H_i \right)^{-1} H_i^t k(X_i,X_i)^{-1} Z(X_i) \pointacc
\]
\nicolas{The predictor} $M_{\text{UK},i}(x)$ is a linear function of $Z(X_i)$, satisfies $\mathbb{E}[ M_{\text{UK},i}(x) ] = \mathbb{E}[  Z(x)  ]$ (for all the possible values of $\beta$ in $\mathbb{R}^m$) and has smallest mean square prediction error among all the predictors with these two properties. 

The next proposition provides the linear aggregation of $M_{\text{UK},1}(x), \ldots , M_{\text{UK},p}(x)$ that is unbiased and has the smallest mean square prediction error.
It thus gives an extension of \NK{} to the universal Kriging case.

\begin{proposition} \label{prop:universal:kriging}
Let $M_{\text{UK}}(x) = ( M_{\text{UK},1}(x), \ldots , M_{\text{UK},p}(x) )^t$.
For $i=1,\ldots,p$ let 
\begin{align*}
w_i(x)^t 
& =
h(x)^t \left( H_i^t k(X_i,X_i)^{-1} H_i \right)^{-1} H_i^t k(X_i,X_i)^{-1}
    -
    k(x,X_i) k(X_i,X_i)^{-1} H_i  \left( H_i^t k(X_i,X_i)^{-1} H_i \right)^{-1} H_i^t k(X_i,X_i)^{-1}
    \\
    & ~ ~
    +
    k(x,X_i) k(X_i,X_i)^{-1}.
\end{align*}
Let $K_{\text{UK},M}(x)$ be the $p \times p$ matrix defined by, for $i,j=1,\ldots,p$,
\begin{align*}
    (K_{\text{UK},M}(x))_{i,j}
    & =
    w_i(x)^t k(X_i,X_j) w_j(x).
\end{align*}
Let $k_{\text{UK},M}(x)$ be the $p \times 1$ vector defined by, for $i=1,\ldots,p$,
\begin{align*}
    (k_{\text{UK},M}(x))_{i}
    & =
    w_i(x)^t k(X_i,x).
\end{align*}
Let 
\[
\hat{m}_{\text{UK},M}(x)
= 
(\mathrm{1}_p^t K_{\text{UK},M}(x)^{-1} \mathrm{1}_p )^{-1}
\mathrm{1}_p^t K_{\text{UK},M}(x)^{-1}
M_{\text{UK}}(x),
\]
with $\mathrm{1}_p$ the $p \times 1$ vector with entries equal to one.
Let then
\begin{equation} \label{eq:UK:interpretation}
M_{\agg,\text{UK}}(x) = 
\hat{m}_{\text{UK},M}(x)
+
k_{\text{UK},M}(x)^t
K_{\text{UK},M}(x)^{-1}
\left(
M_{\text{UK}}(x) 
-
\hat{m}_{\text{UK},M}(x)
\mathrm{1}_p
\right).
\end{equation}
Then $M_{\agg,\text{UK}}(x)$ is a linear function of $M_{\text{UK},1}(x),\ldots,M_{\text{UK},p}(x)$, satisfies $\mathbb{E}[ M_{\agg,\text{UK}}(x) ] = \mathbb{E}[  Z(x)  ]$ (for all the possible values of $\beta$ in $\mathbb{R}^m$) and has smallest mean square prediction error among all the predictors with these two properties. The vector of aggregation weights is
\begin{align*}
\alpha_{\agg,\text{UK}}(x)^t
 = &
(\mathrm{1}_p^t K_{\text{UK},M}(x)^{-1} \mathrm{1}_p )^{-1}
\mathrm{1}_p^t K_{\text{UK},M}(x)^{-1}
\\
& -
k_{\text{UK},M}(x)^t
K_{\text{UK},M}(x)^{-1}
 \mathrm{1}_p
(\mathrm{1}_p^t K_{\text{UK},M}(x)^{-1} \mathrm{1}_p )^{-1}
\mathrm{1}_p^t K_{\text{UK},M}(x)^{-1}
\\
& +
k_{\text{UK},M}(x)^t
K_{\text{UK},M}(x)^{-1}.
\end{align*}
Then $M_{\agg,\text{UK}}(x) = \alpha_{\agg,\text{UK}}(x)^t M_{\text{UK}}(x)$ and
the mean square error is given by
\begin{align} \label{eq:v:agg:UK}
v_{\agg,\text{UK}}(x)
 =
\mathbb{E}
\left[
\left(
M_{\agg,\text{UK}}(x)
- 
Z(x)
\right)^2
\right] 
 =
k(x,x) 
+
\alpha_{\agg,\text{UK}}(x)^t
K_{\text{UK},M}(x)
\alpha_{\agg,\text{UK}}(x)
 -
2 
\alpha_{\agg,\text{UK}}(x)^t
k_{\text{UK},M}(x). 
\end{align}
\end{proposition}

The aggregated predictor $M_{\agg,\text{UK}}(x)$ can be interpreted as a universal Kriging predictor of $Z(x)$, with the ``observations'' $M_{\text{UK},1}(x), \ldots , M_{\text{UK},p}(x)$, and with a constant unknown mean. This is particularly apparent in \eqref{eq:UK:interpretation}, and can be further understood in the proof of Proposition \ref{prop:universal:kriging}. \nicolas{It is worth noting} that the ``observations'' $M_{\text{UK},1}(x), \ldots , M_{\text{UK},p}(x)$ are already themselves universal Kriging predictors. Hence, it turns out that there are two nested steps of universal Kriging predictions when extending the \NK{} predictor to universal Kriging.

\nicolas{Computing $M_{\agg,\text{UK}}(x)$ and $v_{\agg,\text{UK}}(x)$ can be done similarly to what has been proposed in \fb{Sect.}} \ref{subsection:measurement:errors}.
More precisely, the four computational steps are (1) to compute and store the vectors $(w_i(x))_{i=1,\ldots,p}$, (2) to compute and store $(M_{\text{UK},i}(x))_{i=1,\ldots,p}$, with $M_{\text{UK},i}(x) = w_i(x)^t Z(X_i)$ for $i=1,\ldots,p$, (3) to compute and store $K_{\text{UK},M}(x)$ and $(K_{\text{UK},M}(x))^{-1}$ and (4) to compute $M_{\agg,\text{UK}}(x)$ and $v_{\agg,\text{UK}}(x)$.

\nicolas{To analyze the computational complexity and storage requirement, assume that} $X_1,\ldots,X_p$ have cardinalities of order $n/p$ for simplicity. Assume also that $m$ is small compared to $n/p$ and $p$, which is quite realistic in the framework of Kriging with big data, since the number of functions $h_1,\ldots,h_m$ is typically moderate.
Then the computational cost of step (1) is $O(p (n/p)^3)$, the computational cost of step (2) is $O(p (n/p) )$, the computational cost of step (3) is $O( p ^2 (n/p) ^2 + p^3)$ and the computational cost of step (4) is $O( p(n/p) +  p^2)$. As in \fb{Sect.} \ref{subsection:measurement:errors}, the total computational cost is $O( n^3/p^2 + n^2 + p^3 )$ and can reach $O(n^2)$ by taking $p$ of order $n^\beta$ with $\beta \in [1/2,2/3]$. Also as in \fb{Sect.} \ref{subsection:measurement:errors},
the storage cost (including intermediary storage) of steps (1) to (4) is $O((n/p)^2 + p(n/p) + p^2)$ and reaches $O(n)$ by taking $p$ of order $n^{1/2}$.
}

\section{Concluding remarks} \label{section:conclusion}

\nicolas{This article proposes a theoretical analysis of several aggregation procedures recently proposed in the literature, aiming at combining predictions from} Kriging sub-models constructed separately from subsets of a large data set of observations. \nicolas{It is} shown that aggregating the sub-models based only on their conditional variances can yield inconsistent aggregated Kriging predictors. In contrasts, the consistency of the \NK{} procedure \citep{rulliere2018nested}, which explicitly takes into account the correlations between the sub-model predictors, has been proved. The article also shed some light on this procedure, by showing that it provides an exact conditional distribution, for a different Gaussian process prior, and by obtaining bounds on the differences with the exact full Kriging model. \nicolas{Further results on the the efficient computation of conditional covariances have also been presented}, which make possible sampling from the posterior distribution. The impact of the observation assignment to the sub-models has also been investigated, which resulted in some evidence that it is good practice to build them on clusters of observation points. \fb{Finally, the procedure of \citet{rulliere2018nested} has been extended to measurement errors and to universal Kriging, while retaining the same computational complexity and storage requirement.}

Some perspectives remain open. It would be beneficial to improve the aggregation methods of Sect.~\ref{section:non:consistency}, in order to guarantee their consistency while keeping their low computational costs. Finally, the interpretation of the predictor in \citet{rulliere2018nested} as an exact conditional expectation could be the basis of further asymptotic studies, as discussed in Sect.~\ref{section:GP:perspective}.

\bibliographystyle{apalike}
\bibliography{biblio}

\begin{thebibliography}{}

\bibitem[Abrahamsen, 1997]{abrahamsen97review}
Abrahamsen, P. (1997).
\newblock A review of {Gaussian} random fields and correlation functions.
\newblock Technical report, Norwegian Computing Center.

\bibitem[Allard et~al., 2012]{allard2012probability}
Allard, D., Comunian, A., and Renard, P. (2012).
\newblock Probability aggregation methods in geoscience.
\newblock {\em Mathematical Geosciences}, 44(5):545--581.

\bibitem[Bacchi et~al., 2020]{bacchi2020}
Bacchi, V., Jomard, H., Scotti, O., Antoshchenkova, E., Bardet, L., Duluc,
  C.-M., and Hebert, H. (2020).
\newblock Using meta-models for tsunami hazard analysis: An example of
  application for the {French} {Atlantic} coast.
\newblock {\em Frontiers in Earth Science}, 8:41.

\bibitem[Bachoc, 2013]{Bachoc13cross}
Bachoc, F. (2013).
\newblock Cross validation and maximum likelihood estimations of
  hyper-parameters of {Gaussian} processes with model mispecification.
\newblock {\em Computational Statistics and Data Analysis}, 66:55--69.

\bibitem[Bachoc et~al., 2016]{bachoc16improvement}
Bachoc, F., Ammar, K., and Martinez, J. (2016).
\newblock Improvement of code behavior in a design of experiments by
  metamodeling.
\newblock {\em Nuclear science and engineering}, 183(3):387--406.

\bibitem[Bachoc et~al., 2017]{bachoc2017cross}
Bachoc, F., Lagnoux, A., and Nguyen, T. M.~N. (2017).
\newblock Cross-validation estimation of covariance parameters under
  fixed-domain asymptotics.
\newblock {\em Journal of Multivariate Analysis}, 160:42--67.

\bibitem[Banerjee et~al., 2008]{banerjee2008gaussian}
Banerjee, S., Gelfand, A.~E., Finley, A.~O., and Sang, H. (2008).
\newblock {Gaussian} predictive process models for large spatial data sets.
\newblock {\em Journal of the Royal Statistical Society: Series B (Statistical
  Methodology)}, 70(4):825--848.

\bibitem[{Cao} and {Fleet}, 2014]{caoGPOE}
{Cao}, Y. and {Fleet}, D.~J. (2014).
\newblock {Generalized Product of Experts for Automatic and Principled Fusion
  of {Gaussian} Process Predictions}.
\newblock In {\em Modern Nonparametrics 3: Automating the Learning Pipeline
  workshop at {NIPS}, Montreal}.
\newblock arXiv preprint arXiv:1410.7827.

\bibitem[Chevalier and Ginsbourger, 2013]{chevalier2013}
Chevalier, C. and Ginsbourger, D. (2013).
\newblock Fast computation of the multi-points expected improvement with
  applications in batch selection.
\newblock In {\em Learning and Intelligent Optimization}, pages 59--69.
  Springer.

\bibitem[Chiles and Delfiner, 2009]{chiles2009geostatistics}
Chiles, J.-P. and Delfiner, P. (2009).
\newblock {\em Geostatistics: modeling spatial uncertainty}, volume 497.
\newblock John Wiley \& Sons.

\bibitem[Chil{\`e}s and Desassis, 2018]{chiles2018fifty}
Chil{\`e}s, J.-P. and Desassis, N. (2018).
\newblock Fifty years of {Kriging}.
\newblock In {\em Handbook of mathematical geosciences}, pages 589--612.
  Springer, Cham.

\bibitem[Cressie, 1990]{cressie1990origins}
Cressie, N. (1990).
\newblock The origins of kriging.
\newblock {\em Mathematical geology}, 22(3):239--252.

\bibitem[Cressie, 1993]{Cre1993}
Cressie, N. (1993).
\newblock {\em Statistics for spatial data}.
\newblock J. Wiley.

\bibitem[Cressie and Johannesson, 2008]{cressie2008fixed}
Cressie, N. and Johannesson, G. (2008).
\newblock Fixed rank {Kriging} for very large spatial data sets.
\newblock {\em Journal of the Royal Statistical Society: Series B (Statistical
  Methodology)}, 70(1):209--226.

\bibitem[Datta et~al., 2016]{datta2016hierarchical}
Datta, A., Banerjee, S., Finley, A.~O., and Gelfand, A.~E. (2016).
\newblock Hierarchical nearest-neighbor {Gaussian} process models for large
  geostatistical datasets.
\newblock {\em Journal of the American Statistical Association},
  111(514):800--812.

\bibitem[Davis and Curriero, 2019]{davis2019development}
Davis, B.~J. and Curriero, F.~C. (2019).
\newblock Development and evaluation of geostatistical methods for
  non-{Euclidean}-based spatial covariance matrices.
\newblock {\em Mathematical Geosciences}, 51(6):767--791.

\bibitem[Deisenroth and Ng, 2015]{deisenroth2015}
Deisenroth, M.~P. and Ng, J.~W. (2015).
\newblock Distributed {Gaussian} processes.
\newblock {\em Proceedings of the 32nd International Conference on Machine
  Learning, Lille, France. JMLR: W\&CP volume 37}.

\bibitem[Finley et~al., 2009]{finley2009improving}
Finley, A.~O., Sang, H., Banerjee, S., and Gelfand, A.~E. (2009).
\newblock Improving the performance of predictive process modeling for large
  datasets.
\newblock {\em Computational statistics \& data analysis}, 53(8):2873--2884.

\bibitem[Furrer et~al., 2006]{furrer2006covariance}
Furrer, R., Genton, M.~G., and Nychka, D. (2006).
\newblock Covariance tapering for interpolation of large spatial datasets.
\newblock {\em Journal of Computational and Graphical Statistics},
  15(3):502--523.

\bibitem[He et~al., 2019]{he2019query}
He, J., Qi, J., and Ramamohanarao, K. (2019).
\newblock Query-aware {Bayesian} committee machine for scalable {Gaussian}
  process regression.
\newblock In {\em Proceedings of the 2019 SIAM International Conference on Data
  Mining}, pages 208--216. SIAM.

\bibitem[Heaton et~al., 2019]{heaton2019case}
Heaton, M.~J., Datta, A., Finley, A.~O., Furrer, R., Guinness, J., Guhaniyogi,
  R., Gerber, F., Gramacy, R.~B., Hammerling, D., and Katzfuss, M. (2019).
\newblock A case study competition among methods for analyzing large spatial
  data.
\newblock {\em Journal of Agricultural, Biological and Environmental
  Statistics}, 24(3):398--425.

\bibitem[Hensman and Fusi, 2013]{hensman2013}
Hensman, J. and Fusi, N. (2013).
\newblock {Gaussian} processes for big data.
\newblock {\em Uncertainty in Artificial Intelligence}, pages 282--290.

\bibitem[Hinton, 2002]{hinton2002training}
Hinton, G.~E. (2002).
\newblock Training products of experts by minimizing contrastive divergence.
\newblock {\em Neural computation}, 14(8):1771--1800.

\bibitem[Jones et~al., 1998]{jones98efficient}
Jones, D., Schonlau, M., and Welch, W. (1998).
\newblock Efficient global optimization of expensive black box functions.
\newblock {\em Journal of Global Optimization}, 13:455--492.

\bibitem[Kaufman et~al., 2008]{kaufman08covariance}
Kaufman, C.~G., Schervish, M.~J., and Nychka, D.~W. (2008).
\newblock Covariance tapering for likelihood-based estimation in large spatial
  data sets.
\newblock {\em Journal of the American Statistical Association},
  103(484):1545--1555.

\bibitem[Krige, 1951]{krige1951statistical}
Krige, D.~G. (1951).
\newblock A statistical approach to some basic mine valuation problems on the
  witwatersrand.
\newblock {\em Journal of the Southern African Institute of Mining and
  Metallurgy}, 52(6):119--139.

\bibitem[Krityakierne and Baowan, 2020]{krity2020}
Krityakierne, T. and Baowan, D. (2020).
\newblock Aggregated {GP}-based optimization for contaminant source
  localization.
\newblock {\em Operations Research Perspectives}, 7:100151.

\bibitem[Liu et~al., 2018]{liu2018generalized}
Liu, H., Cai, J., Wang, Y., and Ong, Y.-S. (2018).
\newblock Generalized robust {Bayesian} committee machine for large-scale
  {Gaussian} process regression.
\newblock {\em arXiv preprint arXiv:1806.00720}.

\bibitem[{Liu} et~al., 2020]{liu2020}
{Liu}, H., {Ong}, Y., {Shen}, X., and {Cai}, J. (2020).
\newblock When {Gaussian} process meets big data: A review of scalable {GP}s.
\newblock {\em IEEE Transactions on Neural Networks and Learning Systems},
  pages 1--19.

\bibitem[Marrel et~al., 2009]{marrel2009}
Marrel, A., Iooss, B., Laurent, B., and Roustant, O. (2009).
\newblock Calculations of {Sobol} indices for the {Gaussian} process metamodel.
\newblock {\em Reliability Engineering \& System Safety}, 94(3):742--751.

\bibitem[Matheron, 1970]{matheron70theorie}
Matheron, G. (1970).
\newblock {\em La Th\' eorie des Variables R\' egionalis\' ees et ses
  Applications}.
\newblock Fasicule 5 in Les Cahiers du Centre de Morphologie Math{\'e}matique
  de Fontainebleau. Ecole Nationale Sup\'erieure des Mines de Paris.

\bibitem[Putter et~al., 2001]{putter2001effect}
Putter, H., Young, G.~A., et~al. (2001).
\newblock On the effect of covariance function estimation on the accuracy of
  {Kriging} predictors.
\newblock {\em Bernoulli}, 7(3):421--438.

\bibitem[Quinonero-Candela and Rasmussen, 2005]{quinonero2005}
Quinonero-Candela, J. and Rasmussen, C.~E. (2005).
\newblock A unifying view of sparse approximate {Gaussian} process regression.
\newblock {\em The Journal of Machine Learning Research}, 6:1939--1959.

\bibitem[Rasmussen and Williams, 2006]{Rasmussen2006}
Rasmussen, C.~E. and Williams, C.~K. (2006).
\newblock {\em {Gaussian Processes for Machine Learning}}.
\newblock MIT Press.

\bibitem[Roustant et~al., 2012]{roustant12dice}
Roustant, O., Ginsbourger, D., and Deville, Y. (2012).
\newblock {DiceKriging}, {DiceOptim}: Two {R} packages for the analysis of
  computer experiments by {Kriging}-based metamodeling and optimization.
\newblock {\em Journal of Statistical Software}, 51(1).

\bibitem[Rue and Held, 2005]{rue05gaussian}
Rue, H. and Held, L. (2005).
\newblock {\em {Gaussian} Markov random fields, Theory and applications}.
\newblock Chapman \& Hall.

\bibitem[Rulli{\`e}re et~al., 2018]{rulliere2018nested}
Rulli{\`e}re, D., Durrande, N., Bachoc, F., and Chevalier, C. (2018).
\newblock Nested {Kriging} predictions for datasets with a large number of
  observations.
\newblock {\em Statistics and Computing}, 28(4):849--867.

\bibitem[Sacks et~al., 1989]{sacks89design}
Sacks, J., Welch, W., Mitchell, T., and Wynn, H. (1989).
\newblock Design and analysis of computer experiments.
\newblock {\em Statistical Science}, 4:409--423.

\bibitem[Santner et~al., 2013]{santner2013design}
Santner, T.~J., Williams, B.~J., and Notz, W.~I. (2013).
\newblock {\em The design and analysis of computer experiments}.
\newblock Springer Science \& Business Media.

\bibitem[Stein, 2012]{stein2012interpolation}
Stein, M.~L. (2012).
\newblock {\em Interpolation of spatial data: some theory for {Kriging}}.
\newblock Springer Science \& Business Media.

\bibitem[Stein, 2014]{stein14limitations}
Stein, M.~L. (2014).
\newblock Limitations on low rank approximations for covariance matrices of
  spatial data.
\newblock {\em Spatial Statistics}, 8:1--19.

\bibitem[Sun et~al., 2019]{sun2019}
Sun, X., Luo, X.-S., Xu, J., Zhao, Z., Chen, Y., Wu, L., Chen, Q., and Zhang,
  D. (2019).
\newblock Spatio-temporal variations and factors of a provincial pm 2.5
  pollution in eastern china during 2013--2017 by geostatistics.
\newblock {\em Scientific reports}, 9(1):1--10.

\bibitem[Tresp, 2000]{trespBCM}
Tresp, V. (2000).
\newblock A {Bayesian} committee machine.
\newblock {\em Neural Computation}, 12(11):2719--2741.

\bibitem[van Stein et~al., 2015]{van2015optimally}
van Stein, B., Wang, H., Kowalczyk, W., B{\"a}ck, T., and Emmerich, M. (2015).
\newblock Optimally weighted cluster {Kriging} for big data regression.
\newblock In {\em International Symposium on Intelligent Data Analysis}, pages
  310--321. Springer.

\bibitem[Van~Stein et~al., 2020]{van2020cluster}
Van~Stein, B., Wang, H., Kowalczyk, W., Emmerich, M., and B{\"a}ck, T. (2020).
\newblock Cluster-based {Kriging} approximation algorithms for complexity
  reduction.
\newblock {\em Applied Intelligence}, 50(3):778--791.

\bibitem[Vazquez and Bect, 2010a]{vazquez2010convergence}
Vazquez, E. and Bect, J. (2010a).
\newblock Convergence properties of the expected improvement algorithm with
  fixed mean and covariance functions.
\newblock {\em Journal of Statistical Planning and inference},
  140(11):3088--3095.

\bibitem[Vazquez and Bect, 2010b]{vazquez10pointwise}
Vazquez, E. and Bect, J. (2010b).
\newblock Pointwise consistency of the {Kriging} predictor with known mean and
  covariance functions.
\newblock In {\em mODa 9 (Model-Oriented Data Analysis and Optimum Design)
  Springer}.

\bibitem[Ying, 1991]{Ying91}
Ying, Z. (1991).
\newblock Asymptotic properties of a maximum likelihood estimator with data
  from a {Gaussian} process.
\newblock {\em Journal of Multivariate Analysis}, 36:280--296.

\bibitem[Zhang and Wang, 2010]{zhang10kriging}
Zhang, H. and Wang, Y. (2010).
\newblock {Kriging} and cross validation for massive spatial data.
\newblock {\em Environmetrics}, 21:290--304.

\bibitem[Zhu and Zhang, 2006]{zhu2006spatial}
Zhu, Z. and Zhang, H. (2006).
\newblock Spatial sampling design under the infill asymptotic framework.
\newblock {\em Environmetrics: The official journal of the International
  Environmetrics Society}, 17(4):323--337.

\end{thebibliography}

\appendix

\section{Proof of Proposition~\ref{prop:nonconsistency}} \label{app:nonconsistency}

For $v \in \R^m$, we let $|v| = \max_{i=1,\ldots,m} |v_i|$ and $B(v,r) = \{w \in \R^m,|v-w| \leq r\}$.

Let $x_0,\bar{x} \in D$, $r_{x_0}>0$ and $r_{\bar{x}}>0$ be fixed and satisfy $B(x_0,r_{x_0}) \subset D$, $B(\bar{x},r_{\bar{x}}) \subset D$, $B(x_0,r_{x_0}) \cap B(\bar{x},r_{\bar{x}}) = \varnothing$ and $k(x_0,\bar{x}) >0$. [The existence is implied by the assumptions of the proposition.]
By continuity of $k$, $r_{x_0} >0$ and $r_{\bar{x}} >0$ can be selected small enough so that, with some fixed $\epsilon_2 >0$ and $\delta_1 >0$, for $v \in B( x_0 , r_{x_0} )$ and $w \in B( \bar{x} , r_{\bar{x}} )$,
$| v - w | \geq \delta_1$,
 $ k(x_0,x_0) /2 \leq k(v,v) \leq 2k(x_0,x_0)$, $ k(\bar{x},\bar{x}) /2 \leq k(w,w) \leq 2k(\bar{x},\bar{x})$ and
\begin{equation} \label{eq:new:cond:var:v:w}
k(v,v) - \frac{k(v,w)^2}{k(w,w)}
\leq k(v,v) - \epsilon_2.
\end{equation}

For $\delta >0$, let
\[
V(\delta) = \inf_{n \in \mathds{N}} \inf_{\substack{x_0,x_1,...,x_n \in D; \\ \forall i=1,...,n, |x_i - x_0| \geq \delta}} \Var{ Y(x_0) | Y(x_1),...,Y(x_n) }.
\]
Then $V(\delta)>0$ because of the NEB, by continuity of $k$ and by compacity.

Consider a decreasing sequence $\delta_n$ of non-negative numbers such that $\delta_n \to_{n \to \infty} 0$, and which will be specified below.
There exists a sequence $(u_n)_{n \in \mathds{N}} \in D^{\mathds{N}}$, composed of pairwise distinct elements, such that $\lim_{n \to \infty} \sup_{x \in D}\min_{i=1,...,n} | u_{i} - x | = 0$, and such that for all $n$, 
\[
\inf_{ \substack{1 \leq i,j \leq n \\ i \neq j \\ u_i,u_j \in B(x_0,r_{x_0} )}} |u_i - u_j| \geq 4 \delta_n.
\]
Such a sequence indeed exists from Lemma~\ref{lemma:the:sequence} below.

Consider then the sequence $(w_n)_{n \in \mathds{N}} \in D^{\mathds{N}}$ such that for all $n$, $w_n = \bar{x} -(r_{\bar{x}}/(1+n)) e_1$ with $e_1=(1,0,...,0)$. We can assume furthermore that $\{u_n \}_{n \in \mathds{N}}$ and $\{w_n \}_{n \in \mathds{N}}$ are disjoint (this holds almost surely with the construction of Lemma~\ref{lemma:the:sequence} for $(u_n)$).

Let us now consider two sequences of integers $p_n$ and $k_n$ with $k_n \to \infty$ and $p_n \to \infty$ to be specified later. Let $C_n$ be the largest natural number $m$ satisfying $m (p_n-1) < n$. Let $X = (X_1,...,X_{p_n})$ be defined by, for $i=1,...,k_n$, $X_i = ( u_j)_{ j=(i-1)C_n + 1,...,i C_n }$; for $i=k_n+1,...,p_n-1$, $X_i = ( w_j)_{j=(i-k_n-1) C_n + 1,...,(i - k_n) C_n }$; and
$X_{p_n} = ( w_j)_{j=(p_n-k_n-1) C_n + 1,...,n- k_n C_n }$. With this construction, note that $X_{p_n}$ is nonempty. Furthermore, the sequence of vectors $X = (X_{1},...,X_{p_n})$, indexed by $n \in \mathbb{N}$, defines a triangular array of observation points satisfying the conditions of the proposition.

Let us discuss the construction of
$(u_n)_{n \in \mathbb{N}}$, $(w_n)_{n \in \mathbb{N}}$, $k_n$, $C_n$ and $p_n$ more informally. The sequence $(u_n)_{n \in \mathbb{N}}$ is dense in $D$, and $X_1,\ldots,X_{k_n}$ are composed by the $k_n C_n$ first points of this sequence. Then, $X_{k_n+1},\ldots,X_{p_n}$ are composed by the $n - C_n k_n$ first points of the sequence $(w_n)_{n \in \mathbb{N}}$, which is concentrated around $\bar{x}$. We will let $k_n / p_n \to 0$ so that the majority of the groups in $X$ contain points of $(w_n)_{n \in \mathbb{N}}$, so that they do not contain relevant information on the values of $Y$ on $B(x_0,r_{x_0})$ and yield an inconsistency of the aggregated predictor $M_{\agg,n}$ on $B(x_0,r_{x_0})$.

Coming back to the proof, observe that $\inf_{i \in \mathds{N}} \inf_{x \in B(x_0,r_{x_0})} |w_i - x| \geq \delta_1$ and let $\epsilon_1 = V(\delta_1) >0$. Then, we have for all $n \in \mathds{N}$,
for all $x \in B(x_0,r_{x_0})$,
and for all $k=k_n+1,...,p_n$, since then $X_k$ is nonempty and only contains elements $w_i \in B(\bar{x},r)$, from~\eqref{eq:new:cond:var:v:w},
\begin{equation} \label{eq:with:epsun:epsdeux}
\epsilon_1 \leq \vv_k(x) \leq k(x,x) - \epsilon_2.
\end{equation}
Let $\mathcal{E}_n = \{x \in B(x_0,r_{x_0}) ; \min_{i=1,\ldots,n} | x - u_i | \geq \delta_n \}$ and let $x \in \mathcal{E}_n$.
Since $x$ is not a component of $X$, we have $\vv_k(x) >0$ for all $k$. Also $\vv_{p_n}(x) < k(x,x)$ from~\eqref{eq:with:epsun:epsdeux}. Hence, $M_{\agg,n}(x)$ is well-defined.

For two random variables $A$ and $B$, we let $||A-B|| = (\Esp{(A-B)^2})^{1/2}$. Let, for $x \in \mathcal{E}_n$,
\begin{eqnarray*}
R(x) & = & \left| \left | \sum_{k=1}^{k_n} \alpha_{k,n}( \vv_1(x),...,\vv_{p_n}(x),\vv_{prior}(x) )  M_k(x) \right| \right|.
\end{eqnarray*}

Then, from the triangular inequality, and since, from the law of total variance, $|| M_k(x) || \leq || Y( x ) || = \vv_{prior}(x)^{1/2}$ we have, with $\mathcal{V} =  \{ k(x,x) ; x \in B(x_0,r(x_0)) \} $,
\begin{eqnarray*}
R(x) & \leq & \frac{
\sum_{k=1}^{k_n} a( \vv_k(x) , \vv_{prior}(x) ) \sqrt{\vv_{prior}(x)}
}{
\sum_{l=1}^{p_n} b( \vv_l(x) , \vv_{prior}(x) )
}  \\
& \leq &
\frac{
k_n \sup_{v \in \mathcal{V}, V(\delta_n) \leq s^2 \leq v} a( s^2,v ) \sqrt{v}
}{
(p_n - k_n) \inf_{v \in \mathcal{V}, \epsilon_1 \leq s^2 \leq v - \epsilon_2} b( s^2,v )
},
\end{eqnarray*}
where the last inequality is obtained from~\eqref{eq:with:epsun:epsdeux} and the definition of $\delta_n$ and $V(\delta)$.

Let now for $\delta>0$, $s(\delta) = \sup_{v \in \mathcal{V},V(\delta) \leq s^2 \leq v } a( s^2 , v )$.
Since $a$ is continuous and since $V(\delta) >0$, we have that $s(\delta)$ is finite. Hence, we can choose a sequence $\delta_n$ of positive numbers such that $\delta_n \to_{n \to \infty} 0$ and $s(\delta_n) \leq \sqrt{n}$ (for instance, let $\delta_n = \inf \{ \delta \geq n^{-1/2}; s(\delta) \leq n^{1/2} \}$). Then, we can choose $p_n = n^{4/5}$ and $k_n = n^{1/5}$. Then, for $n$ large enough
\[
\frac{k_n}{p_n - k_n} s(\delta_n) \leq 2 n^{-3/5} \sqrt{n} \to_{n \to \infty} 0.
\]
Hence, since
\[
\frac{ \sup_{v \in \mathcal{V}} \sqrt{v} }{ \inf_{v \in \mathcal{V},\epsilon_1 \leq s^2 \leq v - \epsilon_2 }  b(s^2,v) }
\]
is a finite constant, as $b$ is positive and continuous on $\mathring{\Delta}$, we have that $ \sup_{x \in \mathcal{E}_n} R(x) \to_{n \to \infty} 0$. As a consequence, we have from the triangular inequality, for $x \in \mathcal{E}_n$
\begin{align*}
|| Y(x) - M_{\agg,n}(x) || 
\geq &
||Y(x) -  \sum_{k=k_n+1}^{p_n} \alpha_{k,n}( \vv_1(x),...,\vv_{p_n}(x) , \vv_{prior}(x) ) M_k(x) ||
 \\
 & -
|| \sum_{k=k_n+1}^{p_n} \alpha_{k,n}( \vv_1(x),...,\vv_{p_n}(x) , \vv_{prior}(x) ) M_k(x) - M_{\agg,n}(x)  ||  
 \\
 \geq &
\inf_{x \in \mathcal{E}_n}
\left| \left|
Y(x) -
\sum_{k=k_n+1}^{p_n} \alpha_{k,n}( \vv_1(x),...,\vv_{p_n}(x) , \vv_{prior}(x) ) M_k(x)
\right| \right|
\\
& -
 \sup_{x \in \mathcal{E}_n} R(x).
\end{align*}
Since $X_{k_n+1},...,X_{p_n}$ are composed only of elements of $\{ w_i \}_{i \in \mathds{N}}$, we obtain
\[
\liminf_{n \to \infty} 
\inf_{x \in \mathcal{E}_n}
|| Y(x) - M_{\agg,n}(x) ||
\geq V(\delta_1) > 0.
\]
Hence, there exist fixed $n_0 \in \mathbb{N}$ and $A >0$ so that for $n \geq n_0$, $|| Y(x) - M_{\agg,n}(x) || \geq A$.
Hence, we have, for $n \geq n_0$
\begin{align*}
\int_{D}
\Esp{\left(Y(x) - M_{\agg,n}(x)\right)^2}
dx
\geq  &
\int_{\mathcal{E}_n}
\Esp{\left(Y(x) - M_{\agg,n}(x)\right)^2}
\\
\geq &
\int_{\mathcal{E}_n}
A^2
dx.
\end{align*}
Hence, it remains to show that the limit inferior of the volume of $\mathcal{E}_n$ is not zero in order to show~\eqref{eq:int:pred:large}. 
Let $N_n$ be the integer part of $r_{x_0} / 4 \delta_n$. Then, 
the ball $B(x_0,r_{x_0})$ contains $(2 N_n)^d$ disjoint balls of the form $B(a,4 \delta_n)$ with $a \in B(x_0,r_{x_0})$. If one of these balls $B(a,4 \delta_n)$ does not intersect with $(u_i)_{i=1\ldots,n}$, then we can associate to it a ball of the form $B(s_a,\delta_n) \subset B(a,4 \delta_n) \cap \mathcal{E}_n$. If one of these balls $B(a,4 \delta_n)$ does intersect with one $u_j \in \{u_i\}_{i=1\ldots,n}$, then we can find a ball $B(s_a, \delta_n/2  ) \subset ( B(u_j, 2 \delta_n) \backslash B(u_j, \delta_n) ) \cap B(a,4 \delta_n) \cap \mathcal{E}_n$. Hence, we have found $(2 N_n)^d$ disjoint balls with radius $\delta_n/2$ in $\mathcal{E}_n$.
Hence $\mathcal{E}_n$ has volume at least $2^d ((r_{x_0} / 4 \delta_n) - 1)^d \delta_n^{d} $ which has a strictly positive limit inferior. Hence,~\eqref{eq:int:pred:large} is proved. 

Finally, 
if 
$ \Esp{\left(Y(x_0) - M_{\agg,n}(x_0)\right)^2} \to 0$  as $n \to \infty$ for almost all $x_0 \in D$, then 
\[
\int_D \max \left( \Esp{\left(Y(x_0) - M_{\agg,n}(x)\right)^2} , 1 \right) dx \to_{n \to \infty} 0
\]
 from the dominated convergence theorem.
This is contradictory with the proof of~\eqref{eq:int:pred:large}.
Hence,~\eqref{eq:pred:large} is proved.
$\qquad$

\begin{lemma} \label{lemma:the:sequence}
There exists a sequence $(u_n)_{n \in \mathds{N}} \in D^{\mathds{N}}$, composed of pairwise distinct elements, such that
\begin{equation} \label{eq:in:lemma:dense}
\lim_{n \to \infty} \sup_{x \in D}\min_{i=1,...,n} | u_{i} - x | = 0,
\end{equation}
and such that for all $n$, 
\begin{equation} \label{eq:min:dist:ui}
\inf_{ \substack{1 \leq i,j \leq n \\ i \neq j \\ u_i,u_j \in B(x_0,r_{x_0} )}} |u_i - u_j| \geq 4 \delta_n.
\end{equation}
\end{lemma}
\begin{proof}
Such a sequence can be constructed, for instance, by the following random procedure. Let $D \subset B(0,R)$ for $R>0$ large enough. Define $u_1 \in D$ arbitrarily.
For $n=1,2,\ldots$: (1) if the set $ \mathcal{S}_n = \{u \in B(x_0,r_{x_0} ) ; \min_{i=1,\ldots,n} |u-u_i| > 4 \delta_{n+1} \}$ is non-empty, sample $u_{n+1}$ from the uniform distribution on $\mathcal{S}_n$. (2) If $\mathcal{S}_n$ is empty, sample $\tilde{u}_{n+1}$ from the uniform distribution on $B(0,R) \backslash B(x_0,r_{x_0} )$, and set $u_{n+1}$ as the projection of $\tilde{u}_{n+1}$ on $D \backslash B(x_0,r_{x_0} )$. One can see that~\eqref{eq:min:dist:ui} is satisfied by definition. Furthermore, one can show that~\eqref{eq:in:lemma:dense} holds almost surely.
Indeed, let $x \in B(x_0,r_{x_0} )$ and $\epsilon>0$,
and assume that with non-zero probability $B(x,\epsilon) \cap \{ u_i\}_{i \in \mathbb{N}} = \varnothing$. 
Then, the case (1) occurs infinitely often and, for each $i$ for which the case (1) occurs, there is a probability at least $ \epsilon^d / (2 r_{x_0})^d $ that $u_i \in B(x,\epsilon)$ (when $4 \delta_n \leq \epsilon / 2$). This yields a contradiction. Hence, for all $x \in B(x_0,r_{x_0} )$ and $\epsilon>0$, almost surely, $B(x,\epsilon) \cap \{ u_i\}_{i \in \mathbb{N}} \neq \varnothing$. We show similarly, for all $x \in D \backslash B(x_0,r_{x_0} )$ and $\epsilon>0$, almost surely, $B(x,\epsilon) \cap \{ u_i\}_{i \in \mathbb{N}} \neq \varnothing$. This show that~\eqref{eq:in:lemma:dense} holds almost surely. Hence, a fortiori, there exists a sequence $(u_n)_{n \in \mathds{N}} \in D^{\mathds{N}}$ satisfying the conditions of the lemma. 
  \qedperso
 \end{proof}

\begin{remark} \label{remark:clustering}
Consider the case $d=1$.
The proof of Proposition~\ref{prop:non:consistency} can be modified so that the partition $X_1,\ldots,X_{p_n}$ also satisfies $x \leq x'$ for any $x \in X_i$, $x' \in X_j$, $1 \leq i < j \leq p_n$. To see this, consider the same $X$ as in this proof. Let $X_1,\ldots,X_{p_n}$ have the same cardinality as in this proof, and let the $C_n$ smallest elements of $X$ be affected to $X_1$, the next $C_n$ smallest be affected to $X_2$ and so on. Then, one can show that there are at most $k_n+2$ groups containing elements of $(u_i)_{i \in \mathbb{N}} \cap B(x_0 , r_{x_0})  $ and at least $p_n - k_n -2$ groups containing only elements of $B(\bar{x},r_{\bar{x}})$. From these observations,~\eqref{eq:int:pred:large} and~\eqref{eq:pred:large} can be proved similarly as in the proof of Proposition~\ref{prop:non:consistency}. 
\end{remark}

\section{Proof of Proposition~\ref{prop:consistency}} \label{app:consistency}

Because $D$ is compact we have $\lim_{n \to \infty} \sup_{x \in D} \min_{i=1,...,n} || x_{ni} - x || = 0$. Indeed, if this does not hold, there exists $\epsilon>0$ and a subsequence $\phi(n)$ such that $\sup_{x \in D} \min_{i=1,...,\phi(n)} || x_{\phi(n)i} - x || \geq 2 \epsilon$. Hence, there exists a sequence, $x_{\phi(n)} \in D$ such that $\min_{i=1,...,\phi(n)} || x_{\phi(n)i} - x_{\phi(n)} || \geq \epsilon$. Since $D$ is compact, up to extracting a further subsequence, we can also assume that
$x_{\phi(n)} \to_{n \to \infty} x_{lim}$ with $x_{lim} \in D$. This implies that for all $n$ large enough, $\min_{i=1,...,\phi(n)} || x_{\phi(n)i} - x_{lim} || \geq \epsilon / 2$, which is in contradiction with the assumptions of the proposition.

Hence there exists a sequence of positive numbers $\delta_n$ such that $\delta_n \to_{n \to \infty} 0$ and such that for all $x \in D$ there exists a sequence of indices $i_n(x)$ such that $i_n(x) \in \{1,...,n\}$ and $||x - x_{n i_n(x)}|| \leq \delta_n$. There also exists a sequence of indices $j_n(x)$ such that $x_{ni_n(x)}$ is a component of $X_{j_n(x)}$. With these notations we have, since $M_1(x)$,..., $M_{p_n}(x)$, $M_{\agg}(x)$ are linear combinations with minimal square prediction errors,
\begin{eqnarray} \label{eq:nearest:neighboor}
\sup_{x \in D} \Esp{ \left( Y(x) - M_{\agg}(x) \right)^2}
& \leq &
\sup_{x \in D} \Esp{\left(Y(x) - M_{j_n(x)}(x)\right)^2}   \nonumber \\
& \leq & \sup_{x \in D} \Esp{\left(Y(x) - \Esp{Y(x) | Y(x_{ni_n(x)})}\right)^2}.
\end{eqnarray}

In the rest of the proof we essentially show that, for a dense triangular array of observation points, the Kriging predictor that predicts $Y(x)$ based only on the nearest neighbor of $x$ among the observation points has a mean square prediction error that goes to zero uniformly in $x$ when $k$ is continuous. We believe that this fact is somehow known, but we have not been able to find a precise result in the literature.
We have from~\eqref{eq:nearest:neighboor},
\begin{flalign*}
&
\sup_{x \in D} \Esp{
\left(
Y(x) - M_{\agg}(x)
\right)^2
}
& \\
& \leq \sup_{x \in D} \left[ \mathrm{1} \{ k(x_{ni_n(x)},x_{ni_n(x)}) = 0 \} k(x,x) + \mathrm{1} \{ k(x_{ni_n(x)},x_{ni_n(x)}) > 0 \}
\left( k(x,x) - \frac{k(x,x_{ni_n(x)})^2}{k(x_{ni_n(x)},x_{ni_n(x)})} \right) \right] & \\
& \leq \sup_{\substack{x,t \in D; \\ ||x-t|| \leq \delta_n}}
\left[ \mathrm{1} \{ k(t,t) = 0 \} k(x,x) + \mathrm{1} \{ k(t,t) > 0 \}\left( k(x,x) - \frac{k(x,t)^2}{k(t,t)} \right) \right]& \\
& =  \sup_{\substack{x,t \in D; \\ ||x-t|| \leq \delta_n}} F(x,t). &
\end{flalign*}
Assume now that the above supremum does not go to zero as $n \to \infty$. Then there exists $\epsilon>0$
and two sub-sequences $x_{\phi(n)}$ and $t_{\phi(n)}$ with values in $D$ such that $x_{\phi(n)} \to_{n \to \infty} x_{lim}$ and $t_{\phi(n)} \to_{n \to \infty} x_{lim}$, with  $x_{lim} \in D$ and such that $F(x_{\phi(n)},t_{\phi(n)}) \geq \epsilon$. If $k(x_{lim},x_{lim}) = 0$ then $F(x_{\phi(n)},t_{\phi(n)}) \leq k(x_{\phi(n)},x_{\phi(n)}) \to_{n \to \infty} 0$. If $k(x_{lim},x_{lim}) > 0$ then for $n$ large enough
\begin{eqnarray*}
F(x_{\phi(n)},t_{\phi(n)}) = k(x_{\phi(n)},x_{\phi(n)}) - \frac{k(x_{\phi(n)},t_{\phi(n)})^2}{k(t_{\phi(n)},t_{\phi(n)})}
\end{eqnarray*}
which goes to zero as $n \to \infty$ since $k$ is continuous. Hence we have a contradiction, which completes the proof.

\section{Proofs in Sect.~\ref{section:GP:perspective}}

First notice that denoting $k_\agg(x,x')  = \Cov{\Y_\agg(x), \Y_\agg(x')}$, we easily get for all $x, x' \in D$,
\begin{equation}
\begin{split}
	k_\agg(x,x') &= k(x,x') + 2 \transpose{k_M(x)} K_M^{-1}(x) K_M(x,x') K_M^{-1}(x') k_M(x') \\
	& \qquad- \transpose{k_M(x)} K_M^{-1}(x)k_M(x,x') - \transpose{k_M(x')} K_M^{-1}(x')k_M(x',x).
\end{split}
\label{eq:kagg}
\end{equation}
A direct consequence of~\eqref{eq:kagg} is $k_\agg(x,x) = k(x,x)$, and under the interpolation assumption~\ref{H2}, since $\Y_\agg(X) = \Y(X)$, $k_\agg(X,X) = k(X,X)$.

\begin{proof}[Proof of Proposition~\ref{prop:PriorPosterior}]
The interpolation hypothesis $\M_\agg(X) = \Y(X)$ ensures $\varepsilon'_\agg(X)=0$ so we have
	\begin{equation}
	\begin{split}
		\Esp{\Y_\agg(x) | \Y_\agg(X)} &= \Esp{\Y_\agg(x) | \M_\agg(X)+0} \\
		& = \Esp{\M_\agg(x) | \M_\agg(X)} + \Esp{\varepsilon'_\agg(x)  | \M_\agg(X)} \\
		& = \Esp{g_x(\Y(X)) | \Y(X)} + 0 \\
		& = \M_\agg(x).
	\end{split}
	\end{equation}
	The proof that $v_\agg$ is a conditional variance follows the same pattern:
	\begin{equation}
	\begin{split}
		\Var{\Y_\agg(x) | \Y_\agg(X)} &= \Var{\Y_\agg(x) | \M_\agg(X)} \\
		& = \Var{\M_\agg(x) | \M_\agg(X)} +  \Var{\varepsilon'_\agg(x)  }\\
		& = v_\agg(x).
	\end{split}
	\end{equation}
	  \qedperso
\end{proof}

\begin{proof}[Proof of Proposition \ref{prop:condCovYagg}]
Eq.~\eqref{Eq:CrossCov:Expr1} is the classical expression of Gaussian conditional covariances, based on the fact that $Y_\agg$ is Gaussian. Let us now prove Eq.~\eqref{Eq:CrossCov:Expr2}. For a component $x_k$ of the vector of points $X$, using the interpolation assumption, we have $M_\agg(x_k) = Y(x_k)$ and 
$$\Cov{Y_\agg(x),Y_\agg(x_k)}= \Cov{M_\agg(x)+ \varepsilon_\agg'(x), M_\agg(x_k)} = \Cov{M_\agg(x), Y(x_k)}.$$
Remark that
$\alpha_\agg(x)$ is the $p \times 1$ vector of aggregation weights of different sub-models at point $x$, so that $M_\agg(x)= \transpose{\alpha_\agg(x)}M(x)$ and 
$
k_\agg(x,x_k) = \transpose{\alpha_\agg(x)} \Cov{M(x), Y(x_k)}$.
We thus get
\begin{eqnarray}\label{eq:kaggxXfirst}
k_\agg(x,X) &=& \transpose{\alpha_\agg(x)} \Cov{M(x), Y(X)} \, .
\end{eqnarray}
Under the linearity assumption, there exists a $p \times n$ deterministic matrix $\Lambda(x)$ such that $M(x)=\Lambda(x) Y(X)$. Thus $k_\agg(x,X) = \transpose{\alpha_\agg(x)} \Lambda(x) k(X,X)$. 
As remarked in Sect. \ref{section:rulliere}, because of the interpolation condition, $k_\agg(X,X)=k(X,X)$ and\begin{eqnarray}
k_\agg(x,X) k_\agg(X,X)^{-1} k_\agg(X,x') = \transpose{\alpha_\agg(x)} \Lambda(x) k(X,X) \transpose{\Lambda(x')} \alpha_\agg(x') \, .
\end{eqnarray}
Using $K_M(x,x')=\Cov{M(x),M(x')} = \Lambda(x) k(X,X) \transpose{\Lambda(x')}$, we get 
\begin{eqnarray}
k_\agg(x,X) k_\agg(X,X)^{-1} k_\agg(X,x') = \transpose{\alpha_\agg(x)} K_M(x,x') \alpha_\agg(x') \, .\label{eq:simpli1}
\end{eqnarray}
At last, starting from Eq.~\eqref{Eq:CrossCov:Expr1} and using both Eqs~\eqref{eq:kagg} and~\eqref{eq:simpli1}, we get  Eq.~\eqref{Eq:CrossCov:Expr2}.\\
Finally, the development of $\Esp{\left(Y(x)-M_\agg(x)\right)\left(Y(x')-M_\agg(x')\right)}$ leads to the right hand side of Eq.~\eqref{Eq:CrossCov:Expr2} so that 
\[
\Esp{\left(Y(x)-M_\agg(x)\right)\left(Y(x')-M_\agg(x')\right)} = c_\agg(x,x') 
\]
 and Eq.~\eqref{Eq:CrossCov:Expr3} holds.
  \qedperso
\end{proof}

\section{Proofs in Sect.~\ref{sec:bounds}}

\begin{proof}[Proof of Proposition~\ref{prop:bounds}]
Consider $\Delta(x)$ as defined in Eq.~\eqref{eq:diffaggfull}. From Eq.~\eqref{eq:kaggxXfirst}, using both the linear and the interpolation assumptions, we get $k(x,X)\Delta(x) = \left[ k(x,X) - k_\agg(x,X) \right] k(X,X)^{-1}$. Injecting this result in Eq.~\eqref{eq:diffaggfull}, we have 
\begin{equation}\label{eq:MaMfullFunctionkkagg}
\M_\agg(x)-\M_\full(x)= [k_\agg(x,X)-k(x,X)] k(X,X)^{-1} Y(X)
\end{equation}
 and the first equality holds. From~\eqref{eq:diffaggfull}, we also get $v_\agg(x)-v_\full(x)=k(x,X)k(X,X)^{-1}k(X,x)-k_\agg(x,X)k(X,X)^{-1}k_\agg(X,x)$ and the second equality holds. Note that under the same assumptions, we can also use $k_\agg(X,X)=k(X,X)$ and $k_\agg(x,x)=k(x,x)$ and start from $M_\agg=k_\agg(x,X)k_\agg(X,X)^{-1}Y(X)$ and $v_\agg(x)= k_\agg(x,x)-k_\agg(x,X)k_\agg(X,X)^{-1}k_\agg(X,x)$ to get the same results. 

Let us now show Eq.~\eqref{eq:va-vfullbounds}. The upper bound comes from the fact that $\M_\agg(x)$ is the best linear combination of $M_k(x)$ for $k \in \set{1,\ldots, p}$. The positivity of $v_\agg - v_\full$ can be proved similarly: $M_\agg(x)$ is a linear combination of $Y(x_k)$, $k \in \set{1, \ldots, n}$, whereas $M_\full(x)$ is the best linear combination.
Notice that $v_\agg(x)-v_\full(x) \ge 0$ implies, using Eq.~\eqref{eq:diff}, that $\norm{k_\agg(X,x)}_K\le \norm{k(X,x)}_K$.
Let us show Eq.~\eqref{eq:boundsnorms}.
We get the result starting from Eq.~\eqref{eq:MaMfullFunctionkkagg}, applying Cauchy-Schwartz inequality. The bound on $v_\agg(x)-v_\full(x)$ directly derives from Eq.~\eqref{eq:diff}, using $\norm{k_\agg(X,x)}_K\le \norm{k(X,x)}_K$.

Finally, the classical inequality between $\norm{.}_K$ and  $\norm{.}$ derives from the diagonalization of $k(X,X)$, one can notice that it depends on $n$ and $X$, but it does not depend on the prediction point $x$.
  \qedperso
\end{proof}

\begin{proof}[Proof of Remark~\ref{rem:far}]
Using $\norm{k_\agg(X,x)}_K\le \norm{k(X,x)}_K$, using the equivalence of norms and triangular inequality, assuming that the smallest eigenvalue $\lambda_{\min}$ of $k(X,X)$ is non zero, bounds~\eqref{eq:boundsnorms} in the previous Proposition~\ref{prop:bounds} implies that
\begin{equation}\label{eq:boundsnormsGrossier}
		\accolade{
			\abs{\M_\agg(x)-M_\full(x)} & \le & \frac{2}{\lambda_{\min}}\norm{k(X,x)} \norm{Y(X)} \virguleacc\\
			\abs{v_\agg(x)-v_\full(x)} & \le & \frac{1}{\lambda_{\min}}\norm{k(X,x)}^2 \pointacc
		}
\end{equation}
Noticing that the $\norm{.}_K$ and $\lambda_{\min}$ do not depend on $x$ (although they depend on $X$ and $n$), the result holds.
  \qedperso
\end{proof}

\begin{proof}[Proof of Remark~\ref{rem:pointEquivalence}]
As $\Lambda(x)$ is $n \times n$ and invertible, we have
\[
\transpose{k_M(x)} K_M(x)^{-1} M(x)
= \transpose{k(x,X)} \transpose{\Lambda(x)} ( \Lambda(x) k(X,X) \transpose{ \Lambda(x) } )^{-1} \Lambda(x) Y(x)
= M_{\full}(x), 
\]
and similarly $\vv_{\agg}(x) = \vv_{\full}(x)$.
As $\M_\agg = \M_\full$, we have $\Y_\agg = \M_\full + \varepsilon$ where $\varepsilon$ is an independent copy of $\Y - \M_\full$. Furthermore $\Y = \M_\full + \Y - \M_\full$ where $\M_\full$ and $\Y - \M_\full$ are independent, by Gaussianity, so $\Y_\agg \stackrel{law}{=} \Y$.
  \qedperso
\end{proof}

\section{{Proofs in Sect.~\ref{section:extensions}}}

\begin{proof}[\fb{Proof of Proposition \ref{prop:consistency:noisy}}]
\fb{
Because $M_{\agg,\eta}(x)$ is the best linear predictor of $Y(x)$, for $n \in \mathbb{N}$, we have
\begin{equation} \label{eq:for:proof:noisy:one}
\mathbb{E} \left[
\left(
Y(x) - M_{\agg,\eta}(x)
\right)^2
\right]
\leq 
\mathbb{E} \left[
\left(
Y(x) - M_{\eta,i_n}(x)
\right)^2
\right].
\end{equation}
Let $\epsilon >0$.
Let $N_n$ be the number of points in $X_{i_n}$ that are at Euclidean distance less than $\epsilon$ from $x$. By assumption, $N_n \to \infty$ as $n \to \infty$. Let us write these points as $x_{nj_1},\ldots,x_{nj_{N_n}}$, with corresponding measurement errors $\xi_{j_1},\ldots,\xi_{j_{N_n}}$. Since $M_{\eta,i_n}(x)$ is the best linear unbiased predictor of $Y(x)$ from the elements of $Y(x_{n j_1}) + \xi_{j_1} , \ldots ,Y(x_{n j_{N_n}}) + \xi_{j_{N_n}} $, we have
\begin{equation} \label{eq:for:proof:noisy:two}
\mathbb{E} \left[
\left(
Y(x) - M_{\eta,i_n}(x)
\right)^2
\right]
\leq 
\mathbb{E} \left[
\left(
Y(x) - \frac{1}{N_n} \sum_{a=1}^{N_n} ( Y(x_{nj_a}) + \xi_{j_a} )
\right)^2
\right].
\end{equation}
By independence of $Y$ and $\xi_X$, we obtain
\begin{align*}
    \mathbb{E} \left[
\left(
Y(x) - \frac{1}{N_n} \sum_{a=1}^{N_n} ( Y(x_{nj_a}) + \xi_{j_a} )
\right)^2
\right]
& = 
    \mathbb{E} \left[
\left(
\frac{1}{N_n} \sum_{a=1}^{N_n}
(Y(x) -   Y(x_{nj_a}) )
\right)^2
\right]
+
    \mathbb{E} \left[
\left(
\frac{1}{N_n} \sum_{a=1}^{N_n}  \xi_{j_a} 
\right)^2
\right]
\\
& \leq 
\left(
\max_{a=1,\ldots,N_n}
    \mathbb{E} \left[
(Y(x) -   Y(x_{nj_a}) )^2
\right]
\right)
+
\frac{\sum_{a=1}^{N_n} \eta_a}{N_n^2}.
\end{align*}
The above inequality follows from Cauchy-Schwarz, the fact that $Y$ has mean zero and the independence of $\xi_{j_1},\ldots,\xi_{j_{N_n}}$. We then obtain, since $(\eta_a)_{a \in \mathbb{N}}$ is bounded,
\begin{align*}
\limsup_{n \to \infty} 
\mathbb{E} \left[
\left(
Y(x) - \frac{1}{N_n} \sum_{a=1}^{N_n} ( Y(x_{nj_a}) + \xi_{j_a} )
\right)^2
\right]
& \leq 
\sup_{\substack{ u \in D \\ ||u - x|| \leq \epsilon } }    \mathbb{E} \left[
(Y(x) -   Y(u) )^2
\right]
\\
& = 
\sup_{\substack{ u \in D \\ ||u - x|| \leq \epsilon } }  
\left(
 k(x,x) + k(u,u)
 - 2 k(x,u)
\right).
\end{align*}
From \eqref{eq:for:proof:noisy:one} and \eqref{eq:for:proof:noisy:two}, we have, for any $\epsilon >0$, 
\begin{equation} 
\label{eq:for:proof:noisy:lim:sup}
\limsup_{n \to \infty} 
\mathbb{E} \left[
\left(
Y(x) - M_{\agg,\eta}(x)
\right)^2
\right]
\leq 
\sup_{\substack{ u \in D \\ ||u - x|| \leq \epsilon } }  
\left(
 k(x,x) + k(u,u)
 - 2 k(x,u)
\right). 
\end{equation}
The above display goes to zero as $\epsilon \to 0$ because $k$ is continuous. Hence the $\limsup$ in \eqref{eq:for:proof:noisy:lim:sup} is zero, which concludes the proof.   
  \qedperso
}

\end{proof}

\begin{proof}[\fb{Proof of Lemma \ref{lemma:consistency:noisy:random:points}}]
\fb{
Let $\epsilon >0$. For $n \in \mathbb{N}$, let $N_n$ be the number of points in $\{x_1,\ldots,x_n \}$ that are at Euclidean distance less than $\epsilon$ to $x$. Because $x$ is in the interior of $D$ and because $g >0$ on $D$, we have $p_{\epsilon} = \mathbb{P}( ||x_1 - x  || \leq \epsilon) >0$. Hence from the law of large number, almost surely, for $n$ large enough, $N_n \geq (p_{\epsilon}/2) n $. For each $n \in \mathbb{N}$, the $N_n$ points in $\{x_1,\ldots,x_n \}$ that are at Euclidean distance less than $\epsilon$ to $x$ are partitioned into $p_n$ classes. Hence, one of these classes, say the class $X_{i_n}$,  contains a number of points larger or equal to $ N_n / p_n $. Since $n / p_n$ goes to infinity by assumption, we conclude that the number of points in $X_{i_n}$ at distance less than $\epsilon$ from $x$ goes to infinity, almost surely. This concludes the proof.
\qedperso
}
\end{proof}

\begin{proof}[\fb{Proof of Proposition \ref{prop:non:consistency:noisy}}]
\fb{
The proof is based on the same construction of the triangular array of observation points and of the sequence of partitions as in the proof of Proposition \ref{prop:non:consistency}. We take $x$ as $x_0$ in this proof. 
Only a few comments are needed.

We let $V(\delta)$ be as in the proof of Proposition \ref{prop:non:consistency} and we remark that for any $\delta>0$, for any $r \in \mathbb{N}$, for any Gaussian vector $(U_1,\ldots,U_{r})$ independent of $Y$ and for any $u_0,u_1,\ldots,u_r \in D$ with $||u_i - u_0|| \geq \delta$ for $i=1,\ldots,r$, we have
\[
V[ Y(u_0) |  Y(u_1) + U_{1},\ldots,Y(u_r) + U_{r}]
\geq 
V[ Y(u_0) |  Y(u_1),U_{1},\ldots,Y(u_r) , U_{r}]
=
V[ Y(u_0) |  Y(u_1) ,\ldots,Y(u_r) ]
\geq V(\delta).
\]

We also remark that the triangular array and sequence of partitions of the proof of Proposition \ref{prop:non:consistency} do satisfy the condition of Proposition \ref{prop:consistency:noisy}. Indeed, the first component $X_1$ of the partition, with cardinality $C_n \to \infty$, is dense in $D$. 

We remark that for $k=k_n+1,\ldots,p_n$ (notations of the proof of Proposition \ref{prop:non:consistency}), for any row of $X_k$, of the form $x_{nb}$ with $b \in \{ 1 , \ldots,n\}$, we have
$v_k(x) \leq V[ Y(x) |  Y(x_{nb}) + \xi_{b} ] \leq k(x,x) - k(x,x_{nb})^2 / (k(x_{nb},x_{nb}) + \eta_{b})$. Hence, because $(\eta_i)_{i \in \mathbb{N}}$ is bounded, there is a fixed $\epsilon'_2 >0$ such that for $k=k_n+1,\ldots,p_n$, $\epsilon_1 \leq v_k(x) \leq k(x,x) - \epsilon'_2$, with $\epsilon_1$ as in the proof of Proposition \ref{prop:non:consistency}.

With these comments, the arguments of the proof of Proposition \ref{prop:non:consistency} lead to the conclusion of Proposition \ref{prop:non:consistency:noisy}.
\qedperso
}
\end{proof}

\begin{proof}[\fb{Proof of Proposition \ref{prop:universal:kriging}}]
\fb{
 We can see that $M_{\text{UK},i}(x) = w_i(x)^t Z(X_i) $ for $i=1,\ldots,p$. Hence, for $i,j=1,\ldots,p$,
\[
\Cov{ M_{\text{UK},i}(x) , M_{\text{UK},j}(x) }
=
w_i(x)^t
\Cov{ Z(X_i) , Z(X_j) }
w_j(x)
=
w_i(x)^t
k(X_i,X_j)
w_j(x).
\]
Hence, $\Cov{ M_{\text{UK}}(x) } =  K_{\text{UK},M}(x) $. Furthermore, for $i = 1 , \ldots ,p$,
\[
\Cov{ M_{\text{UK},i}(x) , Z(x) }
=
w_i(x)^t
\Cov{ Z(X_i) , Z(x) }
=
w_i(x)^t
k(X_i,x).
\]
Hence, $\Cov{ M_{\text{UK}}(x) , Z(x) } =  k_{\text{UK},M}(x) $.
 Let
\begin{equation} \label{eq:the:constrained:optim:problem}
\alpha(x)
= 
\underset{
\substack{
\gamma \in \mathbb{R}^p \\ 
\mathbb{E}[ \gamma^t M_{\text{UK}}(x) ] = \mathbb{E}[ Z(x) ] \\
\text{for any value of $\beta$ in $\mathbb{R}^m$}
}
}{
\mathrm{argmin}
}
\mathbb{E}
\left[
\left(
\gamma^t M_{\text{UK}}(x)
-
Z(x)
\right)^2
\right].
\end{equation}
Since $\mathbb{E}[  M_{\text{UK},i}(x) ] = \mathbb{E}[ Z(x) ]$ for $i=1,\ldots,p$ and for any value of $\beta \in \mathbb{R}^m$, the constraint in \eqref{eq:the:constrained:optim:problem} can be written as $\gamma^t \mathrm{1}_p  \mathbb{E}[ Z(x)] = \mathbb{E}[ Z(x) ] $ that is $\gamma^t \mathrm{1}_p = 1$. The mean square prediction error in \eqref{eq:the:constrained:optim:problem} can be written as
\[
k(x,x) + \gamma^t K_{\text{UK},M}(x) \gamma
- 2 \gamma^t  k_{\text{UK},M}(x).
\]
Thus \eqref{eq:the:constrained:optim:problem} becomes
\[
\alpha(x)
= 
\underset{
\substack{
\gamma \in \mathbb{R}^p \\ 
\gamma^t \mathrm{1}_p = 1
}
}{
\mathrm{argmin}
}
\left(
k(x,x) + \gamma^t K_{\text{UK},M}(x) \gamma
- 2 \gamma^t  k_{\text{UK},M}(x)
\right).
\]

We recognize the optimization problem of ordinary Kriging which corresponds to universal Kriging with an unknown constant mean function \citep{sacks89design,chiles2009geostatistics}. Hence, we have 
\[
\alpha(x)^t
M_{\text{UK}}(x)
=
\hat{m}_{\text{UK},M}(x)
+
k_{\text{UK},M}(x)^t
K_{\text{UK},M}(x)^{-1}
\left(
M_{\text{UK}}(x) 
-
\hat{m}_{\text{UK},M}(x)
\mathrm{1}_p
\right),
\]
from for instance \cite{sacks89design,chiles2009geostatistics}. 
Hence we have $\alpha(x)^t M_{\text{UK}}(x)  = 
M_{\agg,\text{UK}}(x)$, the best linear predictor described in Proposition \ref{prop:universal:kriging}.

We can see that $M_{\agg,\text{UK}}(x) = \alpha_{\agg,\text{UK}}(x)^t M_{\text{UK}}(x) $  and that $\alpha_{\agg,\text{UK}}(x) = \did{\alpha(x)}$. Then since $\mathbb{E}[ \alpha_{\agg,\text{UK}}(x)^t
M_{\text{UK}}(x) ] = Z(x)$, from $\Cov{ M_{\text{UK}}(x) } =  K_{\text{UK},M}(x) $ and from $\Cov{ M_{\text{UK}}(x) , Z(x) } =  k_{\text{UK},M}(x) $, we obtain  
\[
\mathbb{E}
\left[
\left(
M_{\agg,\text{UK}}(x)
- 
Z(x)
\right)^2
\right] 
 =
k(x,x) 
+
\alpha_{\agg,\text{UK}}(x)^t
K_{\text{UK},M}(x)
\alpha_{\agg,\text{UK}}(x)
 -
2 
\alpha_{\agg,\text{UK}}(x)^t
k_{\text{UK},M}(x). 
\]
This concludes the proof.
\qedperso{} 
} 
\end{proof}
\end{document}